\documentclass[11pt,leqno]{article}
\usepackage[%
tmargin=1in,bmargin=1in,%
lmargin=1.4in,rmargin=1.4in,%
marginparsep=0.2in, marginparwidth=1.0in%
]{geometry}
\AtBeginDocument{%
	\setlength{\abovedisplayskip}{1.5ex plus 0.3ex minus 0.3ex}%
	\setlength{\abovedisplayshortskip}{1.0ex plus 0.3ex minus 0.3ex}%
	\setlength{\belowdisplayskip}{1.5ex plus 0.3ex minus 0.3ex}%
	\setlength{\belowdisplayshortskip}{1.0ex plus 0.3ex minus 0.3ex}%
}
\usepackage{amscd}
\usepackage{amsfonts}
\usepackage{amsmath}
\usepackage{amssymb}
\usepackage{amsthm}
\usepackage{fancyhdr}
\usepackage{xparse}
\usepackage{physics}
\usepackage{xspace}
\usepackage{titlesec}
\usepackage[dotinlabels]{titletoc}
\usepackage{textcase}
\usepackage[colorlinks=true, pdfstartview=FitV, linkcolor=black,
citecolor=black, urlcolor=blue]{hyperref}
\usepackage[style=alphabetic,maxbibnames=99,backend=bibtex8]{biblatex}

\usepackage{ulem}

\usepackage{cancel}

\usepackage{color}

\definecolor{Red}{rgb}{1.00, 0.00, 0.00}
\definecolor{DarkGreen}{rgb}{0.00, 1.00, 0.00}
\definecolor{Blue}{rgb}{0.00, 0.00, 1.00}
\definecolor{Cyan}{rgb}{0.00, 1.00, 1.00}
\definecolor{Magenta}{rgb}{1.00, 0.00, 1.00}
\definecolor{DeepSkyBlue}{rgb}{0.00, 0.75, 1.00}
\definecolor{DarkGreen}{rgb}{0.00, 0.39, 0.00}
\definecolor{SpringGreen}{rgb}{0.00, 1.00, 0.50}
\definecolor{DarkOrange}{rgb}{1.00, 0.55, 0.00}
\definecolor{OrangeRed}{rgb}{1.00, 0.27, 0.00}
\definecolor{DeepPink}{rgb}{1.00, 0.08, 0.57}
\definecolor{DarkViolet}{rgb}{0.58, 0.00, 0.82}
\definecolor{SaddleBrown}{rgb}{0.54, 0.27, 0.07}
\definecolor{Black}{rgb}{0.00, 0.00, 0.00}
\definecolor{dark-magenta}{rgb}{.5,0,.5}
\definecolor{myblack}{rgb}{0,0,0}
\definecolor{darkgray}{gray}{0.5}
\definecolor{lightgray}{gray}{0.75}

\usepackage[pdftex]{graphicx}
\usepackage{epstopdf}

\newcommand{\zz}{\mathbb{Z}}

\newcommand{\sss}{\mathbb{S}}
\newcommand{\bb}[1]{\mathbb{#1}}

\newcommand{\mc}[1]{\mathcal{#1}}

\newcommand{\fff}{\mathbb{F}}
\newcommand{\simto}{\overset{\sim}{\longrightarrow}}

\DeclareMathOperator{\Hom}{Hom}

\DeclareMathOperator{\THH}{THH}
\DeclareMathOperator{\tmf}{tmf}

\DeclareMathOperator{\LMod}{LMod}

\DeclareMathOperator{\Fil}{Fil}
\DeclareMathOperator{\cofib}{cofib}
\DeclareMathOperator{\ksc}{ksc}
\DeclareMathOperator{\ko}{ko}
\newcommand{\mo}[1]{\operatorname{#1}}

\newcommand{\f}[2]{\frac{#1}{#2}}

\newcommand{\bdot}{\text{\textbullet}}

\newtheorem*{short hand}{theorem name}

\def\makeautorefname#1#2{\expandafter\def\csname#1autorefname\endcsname{#2}}
\def\equationautorefname~#1\null{(#1)\null}
\makeautorefname{footnote}{footnote}%
\makeautorefname{item}{item}%
\makeautorefname{figure}{Figure}%
\makeautorefname{table}{Table}%
\makeautorefname{part}{Part}%
\makeautorefname{appendix}{Appendix}%
\makeautorefname{chapter}{Chapter}%
\makeautorefname{section}{Section}%
\makeautorefname{subsection}{Section}%
\makeautorefname{subsubsection}{Section}%
\makeautorefname{theorem}{Theorem}%
\makeautorefname{corollary}{Corollary}%
\makeautorefname{lemma}{Lemma}%
\makeautorefname{proposition}{Proposition}%
\makeautorefname{pro}{Property}
\makeautorefname{conj}{Conjecture}%
\makeautorefname{definition}{Definition}%
\makeautorefname{notn}{Notation}
\makeautorefname{notns}{Notations}
\makeautorefname{rem}{Remark}%
\makeautorefname{quest}{Question}%
\makeautorefname{example}{Example}%
\theoremstyle{plain}  
\newtheorem{theorem}{Theorem}[section]
\newtheorem{proposition}{Proposition}[section]
\newtheorem{lemma}{Lemma}[section]
\newtheorem{corollary}{Corollary}[section]

\theoremstyle{definition}
\newtheorem{definition}{Definition}[section]
\newtheorem{example}{Example}[section]
\newtheorem{remark}{Remark}[section]

\makeatletter
\let\c@corollary=\c@theorem
\let\c@proposition=\c@theorem
\let\c@lemma=\c@theorem
\let\c@definition=\c@theorem
\let\c@example=\c@theorem
\let\c@remark=\c@theorem
\numberwithin{equation}{section}
\makeatother

\usepackage{tikz}

\usetikzlibrary{patterns}
\usepackage{tikz-cd}
\usepackage{spectralsequences}
	\SseqOrientationSideways

\usepackage[T1]{fontenc}
\usepackage[largesc]{newtxtext}
\usepackage{newtxmath}
\usepackage{MnSymbol}

\usepackage[%
cal=cm,       calscaled=0.95,%
bb=boondox,   bbscaled=1.0,%
frak=euler,   frakscaled=1.05,%
]{mathalfa}  

\makeatletter
\g@addto@macro\bfseries{\boldmath}
\makeatother
\usepackage[textsize=footnotesize,backgroundcolor=orange!50]{todonotes}
\usepackage[compress,capitalize]{cleveref}

\let\theoldbibliography\thebibliography
\renewcommand{\thebibliography}[1]{%
	\theoldbibliography{#1}%
	\setlength{\parskip}{0ex}
	\setlength{\itemsep}{0.5ex plus 0.2ex minus 0.2ex}
	\small
}
\apptocmd{\thebibliography}{\raggedright}{}{}

\renewcommand{\title}[1]{\newcommand{\thetitle}{#1}}
\renewcommand{\author}[1]{\newcommand{\theauthor}{#1}}

\renewcommand{\maketitle}{%
	\begin{center}
		{\linespread{1.15}%
			\bfseries\MakeTextUppercase%
			\thetitle\par} \vspace{4.0ex}
		\footnotesize
		{\MakeUppercase \theauthor}
	\end{center}
	\vspace{3.0ex}
	\thispagestyle{fancy}
}

\renewenvironment{abstract}{\noindent\begin{center}\begin{minipage}{0.8\linewidth}\small{\scshape Abstract.}}{\end{minipage}\end{center}}
\newlength{\tagsep}
\makeatletter
\def\fullwidthdisplay{\displayindent\z@ \displaywidth\columnwidth}
\edef\@tempa{\noexpand\fullwidthdisplay\the\everydisplay}
\everydisplay\expandafter{\@tempa}
\makeatother

\titleformat{\section}{\centering}{\textsection\thesection.}{1.5\tagsep}{\scshape}
\titleformat{\subsection}[runin]{}{\fontseries{b}\selectfont\textsection\bfseries\thesubsection.}{1.5\tagsep}{\bfseries}[.]

\titlespacing*{\section}{0pt}{4ex}{\medskipamount}
\titlespacing*{\subsection}{0pt}{\bigskipamount}{0.5em}

\dottedcontents{section}[1.3em]{\vspace{0.25ex}}{1.3em}{0.7em}
\dottedcontents{subsection}[2.0em]{\vspace{0.25ex}}{2.0em}{0.7em}

\title{
	Whitehead filtrations for computations in Topological Hochschild Homology
}

\author{ Logan Hyslop }

\date{mm/dd/yyyy}

\bibliography{CyclicInvarianceCompArxiv.bib}

\begin{document}
		\DeclareDocumentCommand\rart{ g }{%
	{\ar[r, tail]%
		\IfNoValueF {#1} { \ar[r, tail, "#1"]}%
	}%
}
\DeclareDocumentCommand\dart{ g }{%
	{\ar[d, tail]%
		\IfNoValueF {#1} { \ar[d, tail, "#1"]}%
	}%
}
\DeclareDocumentCommand\rarh{ g }{%
	{\ar[r, two heads]%
		\IfNoValueF {#1} { \ar[r, two heads, "#1"]}%
	}%
}
\DeclareDocumentCommand\darh{ g }{%
	{\ar[d, two heads]%
		\IfNoValueF {#1} { \ar[d, two heads, "#1"]}%
	}%
}

\maketitle

	

	%
	%
	%
	%
	%
	%
	%
	\begin{abstract}
	We discuss spectral sequences coming from Whitehead filtrations in the computation of topological Hochschild homology of ring spectra.  Using cyclic invariance, this makes for simple computations of $\THH$ of connective ring spectra $R$ with coefficients in discrete ring spectra.\footnote{Where discrete means Eilenberg-MacLane.}  In particular, we show how to use this to compute $\THH(\tmf,\fff_2)$, and $\THH(\tmf,\zz_{(2)})$, where $\tmf$ denotes the $\bb{E}_\infty$ ring spectrum of topological modular forms.  Then, we obtain a description of $\THH(\ell/v_1^n)$ in terms of $\THH(\ell,\ell/v_1^n)$, where the latter can be computed by results of \cite{angeltveit2009topological}.  We next explain how the methods of this computation generalize to give us information about $\THH(\mo{cofib}(x^k:\Sigma^{k|x|}R\to R))$ for $R$ and $\mo{cofib}(x^k)$ suitably structured connective ring spectra, $k>1$, and $x\in \pi_{*}(R)$ an arbitrary element in positive degree.  Finally, we examine the general framework to describe the topological Hochschild homology of 2-local connective self-conjugate K-theory, $\ksc_2$.
	\end{abstract}
{\small
	\setcounter{tocdepth}{1}
	\tableofcontents
	\vspace{3.0ex}
}
	\section{Introduction}
		\setcounter{section}{1}

\indent\indent Topological Hochschild homology was introduced by B\"{o}kstedt in 1985 as a generalization of ordinary Hochschild homology to general ring spectra, which has lead to many recent advances in algebraic K-theory \cite{nikolaus2018topological}, which in turn, lead to the recent disproof of Ravenel's telescope conjecture \cite{burklund2023ktheoretic}.  In this paper, we continue the program, and provide new computations of topological Hochschild homology for quotients of the Adams summand $\ell$, and of connective self-conjugate $K$-theory $\ksc_2$ at the prime 2.  We also compute the topological Hochschild homology of the connective $\bb{E}_{\infty}$-ring $\tmf$ of topological modular forms \cite{goerss2009topologicalmodularformsaftern} with coefficients in $\fff_2$ and $\zz_{(2)}$, which were previously computed in an unpublished note of Rognes \cite{Rognes14tmf}.  The methods used in principal extend easily to provide computations of $\THH(R,\fff_p)$ whenever $R$ is a (suitably structured) ring spectrum such that $\pi_*(\fff_p\otimes_{R}\fff_p)$ is a quotient of the dual Steenrod algebra, as we will demonstrate in a few select cases in \textsection 3.

The main tools of this paper will be various spectral sequences arising from Whitehead filtrations, used as a means to compute topological Hochschild homology (possibly with coefficients) of connective ring spectra.  In tandem with these filtrations, we will heavily utilize the fact that topological Hochschild homology enjoys a cyclic invariance property, called the Dennis–Waldhausen Morita argument in \cite[~Proposition 6.2]{Blumberg_2012}.  The cyclic invariance property states that if we have a morphism $f:R\to S$ of $\bb{E}_1$-ring spectra, then we have an equivalence $$\THH(R,S)\simeq \THH(S,S\otimes_{R}S).$$  Applying the Whitehead filtration together with this property recovers the Brun spectral sequence, and we get nice comparison maps for computing $\THH$ with coefficients.
	
	For an $\bb{E}_1$-ring spectrum $R$, the topological Hochschild homology of $R$ is defined as $$\THH(R)=R\otimes_{R\otimes_{\sss}R^{op}}R,$$ and the topological Hochschild homology with coefficients in a $R$-bimodule $M$ is similarly given by $$\THH(R,M)=R\otimes_{R\otimes_{\sss}R^{op}}M.$$  Assuming that $R$ is connective, we can apply the Whitehead filtration to $R$ to get a filtered left $R\otimes_{\sss}R^{op}$-module spectrum $\tau_{\geq *}R$.  Applying the functor $$R\otimes_{R\otimes_{\sss}R^{op}}-:\Fil(\LMod_{R\otimes_{\sss}R^{op}})\to\Fil( \LMod_{R})$$ allows us to turn this into a filtered left $R$-module spectrum.  This gives rise to a spectral sequence with signature
	 $$E_1^{s,t}=\THH_{-s}(R,\pi_{t}(R))\implies \THH_{t-s}(R),$$ which we will refer to as the Whitehead spectral sequence.  There is a similar Atiyah-Hirzebruch style spectral sequence computing $\THH(R,M)$ whenever $R$ is a connective $\bb{E}_1$-ring spectrum, and $M$ an $R$-bimodule.  These spectral sequences were considered in the case when $R$ is an $\bb{E}_{\infty}$ ring by H\"{o}nig \cite{H2020}.  One of our main contributions is the extension to the $\bb{E}_n$-algebra case for $n<\infty$.  In section 2, we discuss some basic results on spectral sequences of this type, which will then be applied in the remainder of the paper.
	 
	 In section 3, we show how to use the tools we develop in order to compute $\THH(\tmf,\fff_2)$.  With this in hand, by combining the Brun spectral sequence with a Bockstein spectral sequence, and a careful analysis of some comparisons to the sphere, we are able to compute $\THH(\tmf,\zz_{(2)})$.  The main theorem of section 3 is 
	 \begin{theorem}\label{th1.1}  We have
	 	$$\THH_*(\tmf,\zz_{(2)})=\begin{cases}
	 		\zz_{(2)} & \qquad \text{if } *=0,9,13,22\\
	 		\zz/2^k\zz & \qquad \text{if } *=2r^{k+3}-1, 2^{k+3}r-1+9, 2^{k+3}r-1+13, 2^{k+3}r-1+22,\\
	 		0 & \qquad \text{otherwise},
	 	\end{cases}$$
	 	for all $k>0$ and $r$ odd.
	 \end{theorem}
 Section 4 combines the Atiyah-Hirzebruch spectral sequences with a May-type spectral sequence, as constructed in \cite{angeltveit2015algebraic} (see also \cite{AngeliniKnoll2018},\cite{keenan2020filtration}) , in order to compute the topological Hochschild homology of quotients of $\ell$, $\THH(\ell/v_1^n)$.  Here, $\ell$ denotes the mod $p$ Adams summand for some fixed odd prime $p$, and $v_1\in \pi_{2(p-1)}(\ell)$ generates $\pi_*(\ell)$ as a polynomial algebra over $\zz_{(p)}$.  As we use numerous times throughout, $\ell/v_1^n\simeq\tau_{\leq 2n(p-1)-1}\ell$ inherits an $\bb{E}_{\infty}-\ell$-algebra structure by \cite[~Proposition 7.1.3.15]{HA}.
 \begin{theorem}\label{th1.2}
 	Suppose that $R$ is a connective $\bb{E}_{m}$-ring spectrum for some $m\geq 4$, and $x\in \pi_{*}(R)$ is a positive degree class such that, for some fixed $k>1$, there is an $\bb{E}_3$-$R$-algebra $S$, such that the unique algebra map $R\to S$ fits into a fiber sequence $$\Sigma^{k|x|}R\xrightarrow{x^k}R\to S.$$   Then, there is an equivalence of $\pi_0(R)$-modules $$\THH(S,\pi_0(S))\simeq \THH(R,\pi_0(R))\otimes_{\pi_0(R)}(\pi_0(R)\otimes_{\pi_0(R) \otimes_{R} S}\pi_0(R)).$$
 \end{theorem}
\noindent Under the same hypotheses, this result allows us to construct a spectral sequence with signature $$E_1^{*,*}=\THH_{-*}(R,\pi_*(S))\otimes_{\pi_0(R)}(\bigoplus_{i\geq 0}\pi_0(R)\cdot a_i)\implies \THH_{*}(S),$$ where $a_i$ is a class in bidegree $(-i(2k|x|+2),0)$.  There is a map to this spectral sequence from one with signature $$\THH_{-*}(R,\pi_*(S))\implies \THH_{*}(R,S).$$  Thus, in cases where we have a Leibniz rule on our spectral sequence, the computation reduces to understanding topological Hochschild homology of $R$ with coefficients in a quotient, as well as how differentials act on the classes $a_i$.

As a final sample application, we compute the topological Hochschild homology of 2-local self-conjugate $K$-theory $\ksc_2$ in section \textsection 6.

\noindent\textbf{Notation/Conventions}
\begin{itemize}
	\item We will use the term ``category'' to mean $\infty$-category in the sense of Lurie \cite{lurie2008higher}. 
	\item We write $p$ to denote a fixed odd prime.
	\item The notation $\ell$ will denote the mod $p$ Adams summand, and $\tmf$ will denote the spectrum of topological modular forms.
	\item The divided power algebra on a class $x$ over a ring $R$ will be denoted $\Gamma_{R}[x]$, and is defined by generators $x^{(n)}$ for $n\in\zz_{> 0}$ with relations $x^{(n)}x^{(m)}=\f{(n+m)!}{n!m!}x^{(n+m)}$.
	\item Classical Hochschild homology over a base ring $R$ will be denoted by $\mo{HH}(-/R)$.
\end{itemize}
\noindent\textbf{Acknowledgments.}  I am grateful to Alicia Lima for mentoring me during this project, Peter May for hosting the REU at which this work was done, and Mike Hill for useful conversations including his suggestion to examine $\THH_*(\ksc_2)$.  I am especially grateful to Ishan Levy for recommending that I do the computations that eventually grew into this paper, and for several helpful conversations.  I would also like to thank Jeremy Hahn, Mike Hill, Peter May, and anonymous referees for providing several useful comments on previous versions of the draft.  I would further like to thank Mike Hopkins for pointing out Remark \ref{rem3.2} to me.
		\newpage
\section{Filtered Objects and Spectral Sequences}
\indent We review the theory of filtered objects developed in \cite[~\textsection 1.2.2]{HA}, \cite[~\textsection Appendix B]{burklund2022galois}, simultaneously setting our conventions.  Let $\mc{C}$ be a stable category.  To $\mc{C}$, one can associate the category $\mo{Fil}(\mc{C}):=\mo{Fun}(\zz_{\leq }^{op},\mc{C})$ of filtered objects in $\mc{C}$, i.e. the category of $\zz$-indexed collections of objects $c_i\in\mc{C}$ together with maps $c_i\to c_{i-1}$.  The category of filtered objects in $\mc{C}$ is equipped with an auto-equivalence $(1):\mc{C}\to \mc{C}$, with $c(1)_{n}=c_{n-1}$, and a natural transformation $\tau:(-1)\implies id$ given in degree $n$ by the map $c_{n+1}\to c_n$ induced by the filtration.  Taking the cofiber of $\tau$ induces a functor from $\mo{Fil}(\mc{C})$ to the category of graded objects $\mo{Gr}(\mc{C}):=\mo{Fun}(\zz^{ds},\mc{C})$ of functors from the discrete category $\zz^{ds}$ to $\mc{C}$.  The object $\cofib(\tau:c_{\bdot}(-1)\to c_{\bdot})$ will be denoted $c_{\bdot}^{gr}$, and termed the associated graded object of $c_{\bdot}$.

To any $c\in \mc{C}$ and $i\in\zz$, we associate an object $c^{i}\in \mo{Fil}(\mc{C})$ such that $$c^i_{n}=\begin{cases}
	0 & \qquad \qquad \text{if } i>n,\\
	c & \qquad \qquad \text{otherwise},
\end{cases}$$ with transition maps being 0 when the source is zero and the identity map otherwise.  Suppose that there is an object $1\in\mc{C}$ which defines a notion of ($\zz$-graded) homotopy groups via $$\pi_n(c):=\pi_0(\Hom_{\mc{C}}(\Sigma^{n}1,c)).$$  Then one can define bigraded homotopy groups for filtered objects of $\mc{C}$ via the formula $$\pi_{n,m}(c):=\pi_0(\Hom_{\mo{Fil}(\mc{C})}((\Sigma^n1)^m,c)).$$
The fiber sequence $c(-1)\to c \to c^{gr}$ yields an exact couple, giving rise to a spectral sequence with signature $$E_1^{s,t}:=\pi_{t-s,t}(c^{gr})\implies \pi_{t-s}(\varinjlim c_{\bdot}),$$ which converges to the homotopy groups of the underlying object ``$\varinjlim c_{\bdot}$'' so long as $\varprojlim c_{\bdot}$ vanishes (which in the cases we consider, it always will).  With this grading convention, the $d_r$-differential has bigrading $(r+1,r)$.

If $\mc{C}$ had a symmetric monoidal structure, then so too does $\mo{Fil}(\mc{C})$, via Day convolution \cite[~\textsection2.6]{HA}.  The functor taking the underlying object $\mo{Fil}(\mc{C})\to \mc{C}$ is symmetric monoidal, and with the Day convolution structure on $\mo{Gr}(\mc{C})$, so to is the functor taking the associated graded \cite[~\textsection Appendix B]{burklund2022galois}.

We will use specific $t$-structures on the filtered derived category, generalizing constructions of Beilinson, and investigated in \cite{burklund2022galois} (in this particular form, the definition is a special case of \cite[~Definition 2.4]{lee2023topological}).  
\begin{definition}\label{definition2.1}
	Suppose that $\mc{C}$ is the category of (left) modules over some connective $\bb{E}_1$-ring spectrum (e.g. $\mc{C}=\mo{Sp}$), and denote by $\tau$ the canonical $t$-structure on $\mc{C}$ defined as in \cite[~Proposition 7.1.1.13]{HA}.  Given any $r\in\bb{Q}_{\geq 0}$, there is a $t$-structure $\tau^r$ on $\mo{Fil}(\mc{C})$ (resp. $\mo{Gr}(\mc{C})$) uniquely determined by stipulating that $x_{\bdot}\in\mo{Fil}(\mc{C})$ (resp. $x_{\bdot}\in\mo{Gr}(\mc{C})$) is connective if and only if $x_i$ is an $\lceil ri\rceil$-connective object of $\mc{C}$ for all $i\in\zz$.  In either case, the connective cover $\tau^{r}_{\geq 0}(x_{\bdot})$ has $i$th component $\tau_{\geq \lceil ri\rceil }x_i$.
\end{definition}
By \cite[~Lemma 2.6]{lee2023topological}, if $\mc{C}$ is the category of modules over an $\bb{E}_{\infty}$-ring, then these $t$-structures defined above are compatible with the symmetric monoidal structure.  
\begin{remark}\label{rem2.2}
In particular, the functor $\mo{Sp}\to\mo{Fil}(\mo{Sp})$ is lax symmetric monoidal.  Indeed, it is the composite of the lax symmetric monoidal connective cover functor $\tau_{\geq 0}^{1}$ with the functor taking $c$ to the constant filtered object $\ldots \to c\to c\to c\to\ldots$.  This latter functor is lax symmetric monoidal as it is right adjoint to the symmetric monoidal functor taking a filtered spectrum to the underlying spectrum.
\end{remark}

We now introduce the main spectral sequences that we will be using in this paper. Consider a connective $\bb{E}_{1}$-ring $R$, and an $R$-bimodule $M$, i.e., a left $R\otimes_{\sss} R^{op}$-module. Working in the category $\Fil(\LMod_{R\otimes_{\sss}R^{op}})$, we can take $\tau_{\geq *}M$ to be the Whitehead filtration on $M$, i.e., the filtered spectrum $$\ldots\to \tau_{\geq m}M\to\tau_{\geq m-1}M\to\ldots.$$  Applying the functor $\THH(R;-):\Fil(\LMod_{R\otimes_\sss R^{op}})\to \Fil(\LMod_R)$, we get a filtered object $\THH(R,\tau_{\geq *}M)$.
\begin{definition}\label{def2.3}
 The \textit{Atiyah-Hirzebruch spectral sequence} (AHSS) with coefficients in $M$ is the spectral sequence associated to $\THH(R,\tau_{\geq *}M)$, which has signature $$E_1^{s,t}=\THH_{-s}(R,\pi_{t}(M))\Rightarrow \THH_{t-s}(R,M).$$ In the special case $M=R$, we will refer to the resulting spectral sequence as the \textit{Whitehead spectral sequence} (WSS).
\end{definition} 
In the above definition, $\pi_t(M)$ is treated as a discrete left $R\otimes_{\sss}R^{op}$-module concentrated in degree 0.
\begin{lemma}\label{lemma2.4}
	Let $R$ be a connective $\bb{E}_1$-ring spectrum and $M$ an $R$-bimodule.  The filtered object $\THH(R;\tau_{\geq *}M)$ is complete, so the Atiyah-Hirzebruch spectral sequence converges.
\end{lemma}
\begin{proof}
Since $\THH(R;-)$ is right t-exact, the $n$th filtered piece of $\THH(R;\tau_{\geq *}M)$ is $n$-connective, and the result follows.
\end{proof}
We recall that discrete modules over $R\otimes_{\sss}R^{op}$ in degree 0 are exactly the modules which live in the heart of $\LMod_{R\otimes_{\sss}R^{op}}$ with respect to the standard $t$-structure, and are thus in bijection with $\pi_0(R\otimes_{\sss} R^{op})\simeq \pi_0(R)\otimes_{\zz}\pi_{0}(R)^{op}$-modules.  The AHSS is functorial in $M$, and we also note that:
\begin{proposition}\label{proposition2.5}
	Let $R$ and $S$ be connective $\bb{E}_1$-ring spectra, and $f:R\to S$ a $\bb{E}_1$-ring map.  If $M$ is an $R$-bimodule, then the natural map $\THH(R,M)\to \THH(S,(S\otimes_{\sss}S^{op}\otimes_{R\otimes_{\sss}R^{op}} M))$ induces a map on the associated Atiyah-Hirzebruch spectral sequences.
\end{proposition}
\begin{proof}
It suffices to show that we get a map on the level of filtered objects.  Note that we have a natural factorization $$\tau_{\geq n}M \to (S\otimes_{\sss}S^{op}) \otimes_{R\otimes_{\sss}R^{op}} \tau_{\geq n}M\to \tau_{\geq n}( (S\otimes_{\sss}S^{op})\otimes_{R\otimes_{\sss}R^{op}} M).$$  Applying the functor $\THH(R,-)$ and applying the cyclic invariance property gives a map of filtered objects $$\THH(R,\tau_{\geq *}M)\to \THH(S,(S\otimes_{\sss}S^{op})\otimes_{R\otimes_{\sss}R^{op}}\tau_{\geq *}M).$$ The morphism induced by the second map above gives a map
$$\THH(S,(S\otimes_{\sss}S^{op})\otimes_{R\otimes_{\sss}R^{op}}\tau_{\geq *}M)\to \THH(S, \tau_{\geq*}((S\otimes_{\sss}S^{op})\otimes_{R\otimes_{\sss}R^{op}} M)),$$
and the composite provides our desired map of filtered objects.
\end{proof}

\begin{lemma}\label{lemma2.6}
	Suppose $R$ is a connective $\bb{E}_n$-ring spectrum for some $n\geq 4$.  Then the Atiyah-Hirzebruch spectral sequence with coefficients in $M$ is multiplicative whenever $M$ is an $\bb{E}_2$-$R$-algebra.
\end{lemma}  
\begin{proof}
Suppose $R$ is a $\bb{E}_n$-algebra for some $n\geq 4$.  In particular, $R\otimes_{\sss}R^{op}$ is an $\bb{E}_n$-algebra as well, and the canonical multiplication map $R\otimes_{\sss}R^{op}\to R$ is an $\bb{E}_{n-1}$-algebra map, as given any $R\in\mo{Alg}_{\bb{E}_1}(\mc{C})$ for some symmetric monoidal $\infty$-category $\mc{C}$ (here we take $\mc{C}=\mo{Alg}_{\bb{E}_{n-1}}(\mo{Sp})$, and use \cite[~Theorem 5.1.2.2]{HA}), there is a map $R\otimes R^{op}\simeq R\otimes R\to R$ in $\mc{C}$ determined by the $\bb{E}_1$-multiplication on $R$.  It follows from \cite[~Proposition 7.1.2.6]{HA} that the functor $R\otimes_{R\otimes_{\sss}R^{op}}-:\LMod_{R\otimes_{\sss}R^{op}}\to \LMod_R$ is $\bb{E}_{n-2}$-monoidal, and in particular, $\bb{E}_{2}$-monoidal.  Since the Whitehead filtration is compatible with the monoidal structure \cite[~Lemma 7.1.1.7]{HA}, this gives the claim. 
\end{proof}
Next, we want to be able to find other useful spectral sequences to compare these to that will allow us to figure out some of their differentials.  This is where cyclic invariance comes in.  We work here in the case where $R$ is an $\bb{E}_{\infty}$-ring, but the setup should work more generally with minimal modifications.
\begin{proposition}\label{proposition2.7}
	For $R$ an $\bb{E}_{1}$-ring, and $S$ an $\bb{E}_1$-$R$-algebra, we have an equivalence $\THH(R,S)\simeq \THH(S, S\otimes_{R}S)$.  If $R$ and $S$ are both $\bb{E}_{\infty}$-algebras, then this is an equivalence of $\bb{E}_\infty$-algebras.
\end{proposition}
\begin{proof}
This follows from \cite[~Proposition 6.2]{Blumberg_2012} and \cite[~Lemma 4.8, Remark 4.10]{H2020}.
\end{proof}

\begin{definition}\label{definition2.8}
When we take the Atiyah-Hirzebruch spectral sequence with coefficients in $S\otimes_{R}S$ in a situation as above, the resulting spectral sequence will be called the \textit{Brun-Atiyah-Hirzebruch spectral sequence} (BAHSS).  The BAHSS has signature $$E_1^{s,t}=\THH_{-s}(S,\pi_{t}(S\otimes_{R}S))\Rightarrow \THH_{t-s}(R,S).$$
\end{definition}
The BAHSS was introduced by Brun in \cite{BRUN200029} and studied by H\"{o}ning in \cite{H2020}.

For $R$ an arbitrary connective $\bb{E}_n$-ring spectrum, this already gives us a lot to compare $R$ with.  Namely, it is clear that $\THH(R,R\otimes_{\sss}R^{op})\simeq R$, so applying the Whitehead filtration to the coefficients produces a spectral sequence $E_1^{s,t}=\THH_{-s}(R,\pi_{t}(R))\Rightarrow \pi_{t-s}(R)$.  Furthermore, since $\pi_*(R\otimes_{\sss}R)\to\pi_*(R)$ is split as a map of graded $\pi_*(R)$-modules, the natural map that $R\otimes_{\sss}R^{op}\to R$ induces on the level of spectral sequence is surjective on the $E_1$-page.  Thus, if one could determine the structure of the spectral sequence associated to $\THH(R,R\otimes_{\sss}R^{op})$, one could compute $\THH(R)$, although the structure of the former can get quite complicated in general.  

A similar style spectral sequence that will feature prominently in \textsection 6 is the May-type spectral sequence, studied in detail in \cite{angeltveit2015algebraic}, \cite{AngeliniKnoll2018}, and \cite{keenan2020filtration}.  If $R$ is a connective $\bb{E}_{n}$-ring spectrum, then $\tau_{\geq *}R$ is an $\bb{E}_n$-algebra object in filtered spectra.  In particular, one can apply the construction of topological Hochschild homology to $\tau_{\geq *}R$ internally to filtered spectra, so get a filtered spectrum $$\THH(\tau_{\geq *}R)\simeq \tau_{\geq *R}\otimes_{(\tau_{\geq *}R)\otimes_{\sss^0}(\tau_{\geq *}R^{op})}\tau_{\geq *}R,$$ which allows us to define,
\begin{definition}\label{definition2.9}
The \textit{May-type spectral sequence} for a connective $\bb{E}_n$-ring spectrum (resp. with coefficients in a bimodule $M$) is the spectral sequence associated to the filtered object $\THH(\tau_{\geq *}R)$ (resp. $\THH(\tau_{\geq *}R,\tau_{\geq *}M)$).
\end{definition}
The associated graded takes the form of the topological Hochschild homology of $\pi_{*}R$, which is an $\bb{E}_{n}-\zz-$algebra, making it somewhat more amenable to computations.  We will also consider some slight variants of this spectral sequence later.

\begin{definition}\label{definition2.10}
The \textit{Bockstein algebra} $\zz_{(2)}^{bok}$ is the filtered $\bb{E}_{\infty}\text{-}\zz_{(2)}$-algebra with $(\zz_{(2)}^{bok})_i=\zz_{(2)}$ for all $i$, where the transition map $(\zz_{(2)}^{bok})_i\to (\zz_{(2)}^{bok})_{i-1}$ is the identity if $i\leq 0$, and multiplication by 2 if $i>0$.

Given any $\bb{E}_{\infty}\text{-}\zz_{(2)}$-algebra $R$, we can consider the filtered algebra $R^0\otimes_{\zz_{(2)}^0}\zz_{(2)}^{bok}$, whose associated spectral sequence will be termed the \textit{Bockstein spectral sequence}.  This is a multiplicative spectral sequence which computes $\pi_*(R_2^{\wedge})$.
\end{definition}
The algebra $\zz_{(2)}^{bok}$ lives in the heart of $\mo{Fil}(\mo{Sp})$ for the $\tau^0_{*}$-t-structure from Definition \ref{definition2.1}, and there it is an ordinary commutative algebra as can be checked on the nose, and thus extends to an $\bb{E}_{\infty}$-algebra in $\mo{Fil}(\mo{Sp})$.

Often throughout this paper, particularly when making use of Bockstein spectral sequences, we will encounter ``$\tilde{v}_0$-towers'' in total degree $i$, which could either come from a copy of $\zz_{(2)}$ in the $i$-th homotopy group of the total object, or a $\bb{Q}/\zz_{(2)}$-copy in the $i-1^{\text{st}}$-homotopy group of the total object.  The following argument (sometimes applied repeatedly to a given situation) ensures that we always land in the former case.
\begin{proposition}\label{proposition2.11}
Suppose that $R$ is a connective $\bb{E}_{\infty}$-ring spectrum, such that $\pi_0(R)$ is Noetherian, and $\pi_i(R)$ is finitely generated as a $\pi_0(R)$-module for all $i\geq 0$.  Given connective $R$-modules $M,N\in\mo{Mod}(R)$, with $\pi_i(M)$ and $\pi_i(N)$ finitely generated $\pi_0(R)$-modules for all $i\geq 0$, then $\pi_i(M\otimes_R N)$ is also a finitely generated $\pi_0(R)$-module for all $i\geq 0$.
\end{proposition}
\begin{proof}
In this case, by \cite[~Proposition 7.2.4.17]{HA}, $M$ and $N$ are both almost perfect as $R$-modules.  Consider the category $\mc{C}\subseteq \mo{Mod}(R)$ of all $R$-modules $X$ such that $X\otimes_{R} N$ is almost perfect.  Then it is clear that $\mc{C}$ is closed under suspensions, and using that any fiber sequence $X\to Y \to Z$ extends to a fiber sequence $X\otimes_R N\to Y\otimes_R N\to Z\otimes_R N$, the long exact sequence in homotopy groups (together with Noetherianity of $\pi_0(R)$) shows that $\mc{C}$ satisfies the 2-out-of-3 property, so is a stable subcategory.  Furthermore, $\mc{C}$ is also closed under taking summands, since any submodule of a finite type module over a Noetherian ring is still finite type.  Therefore, $\mc{C}$ is a thick stable subcategory, and by assumption, it contains the unit $R$.  Since the thick stable category generated by $R$ is the category of perfect $R$-modules, $\mc{C}$ contains every perfect $R$-module.  Our claim would follow if we knew that $\mc{C}$ contained all almost perfect $R$-modules, so suppose $M$ is almost perfect, and let $n\in\zz $.  By building $M$ cellularly, i.e., if $M$ is $i$-connective, by beginning with a presentation $\pi_0(R)^{\oplus n}\to \pi_0(R)^{\oplus m}\to \pi_i(M)$, and noting that this gives us a map $\Sigma^{i}\mo{cofib}(R^{\oplus n}\to R^{\oplus m})\to M$ with $i+1$-connective fiber, and repeating with the fiber, we can build a perfect $R$-module $A_n$ with a map $A_n\to M$ and $n+1$-connective fiber.  But then $\pi_k(M\otimes_{R} N)\simeq \pi_k(A_n\otimes_{R}N)$ for all $k\leq n$, and since $A_n$ is perfect, this shows that $\pi_k(M\otimes_{R}N)$ is finitely generated as a $\pi_0(R)$-module, as desired.
\end{proof}

\newpage
\section{Topological Hochschild Homology of $\tmf$ with Coefficients}
\indent\indent In this section, we primarily study the topological Hochschild homology of connective topological modular forms with coefficients in $\fff_2$ and $\zz_{(2)}$, but also include a couple analogous computations at an odd prime $p$.  These computations were studied independently by Bruner-Rognes in a work in progress \cite{Rognes14tmf}.  We begin with $\THH(\tmf,\fff_2)$, where the cyclic invariance condition will do most of the heavy lifting.  As a precursor, let's analyze a simple example.
\begin{example}\label{ex3.1}  Consider the Brun-Atiyah-Hirzebruch spectral sequence with signature $$E_1^{s,t}=\THH_{-s}(\fff_2,\pi_{t}(\fff_2\otimes_{\sss}\fff_2))\Rightarrow\fff_2.$$  By B\"{o}kstedt's computation of $\THH(\fff_2)$, we know that the $E_1$-page is given by $$E_1^{s,t}=\fff_2[\xi_1,\xi_2,\xi_3,\ldots]\otimes_{\fff_2}\fff_2[u],$$ where $\xi_i$ has bidegree $(0,2^i-1)$, and $u$ has bidegree $(-2,0)$.  The differentials must vanish on the classes $\xi_i$ for degree reasons, so by the Leibniz rule, to understand the $d_1$ differential, we only need to understand how it acts on the class $u$.  Since the BAHSS converges to $\fff_2$, the class $u$ cannot be a permanent cycle, and the only differential that can kill it is $d_1(u)=\xi_1$.
	\vskip1pt
	\begin{center}
		\begin{sseqpage}[title ={$E_1$-page}, axes type = center,  classes = {draw = none },
			x range = {-11}{0}, y range = {0}{3}, x axis origin = {1}, y tick gap = {-0.5 cm}, x label = { s }, y label = { t }, x label style={ yshift = -15pt }, y label style={rotate=270, yshift=370pt, xshift=20pt}, class placement transform={scale=2.5}
			]
			\class["1"](0,0)
			\class["u"](-2,0)
			\class["u^2"](-4,0)
			\class["u^3"](-6,0)
			\class["u^4"](-8,0)
			\class["u^5"](-10,0)
			\class["\xi_1"](0,1)
			\class["u\xi_1"](-2,1)
			\class["u^2\xi_1"](-4,1)
			\class["u^3\xi_1"](-6,1)
			\class["u^4\xi_1"](-8,1)
			\class["u^5\xi_1"](-10,1)
			\class["\xi_1^2"](0,2)
			\class["u\xi_1^2"](-2,2)
			\class["u^2\xi_1^2"](-4,2)
			\class["u^3\xi_1^2"](-6,2)
			\class["u^4\xi_1^2"](-8,2)
			\class["u^5\xi_1^2"](-10,2)
			\class["\xi_1^3"](0,3)
			\class["u\xi_1^3"](-2,3)
			\class["u^2\xi_1^3"](-4,3)
			\class["u^3\xi_1^3"](-6,3)
			\class["u^4\xi_1^3"](-8,3)
			\class["u^5\xi_1^3"](-10,3)
			\class["\xi_2"](0,3)
			\class["u\xi_2"](-2,3)
			\class["u^2\xi_2"](-4,3)
			\class["u^3\xi_2"](-6,3)
			\class["u^4\xi_2"](-8,3)
			\class["u^5\xi_2"](-10,3)
			\d2(-2,0)(0,1)
			\d2(-2,1)(0,2)
			\d2(-6,0)(-4,1)
			\d2(-6,1)(-4,2)
			\d2(-10,0)(-8,1)
			\d2(-10,1)(-8,2)
			\d2(-2,2)(0,3)
			\d2(-6,2)(-4,3)
			\d2(-10,2)(-8,3)
		\end{sseqpage}
	\end{center}
	This leaves $$E_2=\fff_2[\xi_2,\xi_3,\ldots ]\otimes_{\fff_2} \fff_2[u^2].$$  Inductively, we find that $E_k=E_{k+1}$ if $k\neq 2^n-1$, $$E_{2^n}=\fff_2[\xi_{n+1},\xi_{n+2},\ldots]\otimes_{\fff_2}\fff_2[u^{2^n}],$$
and $d_{2^{n}-1}(u^{2^{n-1}})=\xi_n$.
	\end{example}
\begin{remark}\label{rem3.2}
Mike Hopkins pointed out that by the Hopkins-Mahowald theorem \cite{MahowaldRS} (see also \cite[~Theorem 1.3]{Blumberg_2010}), the BAHSS in Example \ref{ex3.1} (resp. in Example \ref{example3.3}), is the ``Thomification'' of the mod 2 (resp. mod $p$) Serre spectral sequence for the path-loop fibration $\Omega^2S^3\to P\Omega S^3\to \Omega S^3$.
\end{remark}

\begin{example}\label{example3.3}
In a similar fashion, we can determine the differentials in the BAHSS associated to $\THH(\fff_p,\fff_p\otimes_{\sss}\fff_p)\simeq\fff_p$, which has signature $$E_1^{s,t}=\fff_p[\xi_1,\xi_2,\ldots]\otimes\Lambda[\tau_0,\tau_1,\ldots]\otimes\fff_p[u]\implies \fff_p,$$ where $|\xi_i|= (0,2(p^i-1))$, $|\tau_i|=(0,2p^i-1)$, and $|u|=(-2,0)$.  The differentials in this spectral sequence are entirely determined by the Leibniz rule and the rules $$d_{2p^i-1}(u^{p^i})=\tau_i,\qquad d_{2p^{i-1}(p-1)}(u^{p^{i-1}(p-1)}\tau_{i-1})=\xi_i.$$
\end{example}
\vskip1pt
\begin{center}
	\begin{sseqpage}[title={$E_1$-page for $p=3$}, axes type = center,  classes = {draw = none },
		x range = {-10}{0}, y range = {0}{5}, x axis origin = {1}, y tick gap = {-0.5 cm}, x label = { s }, y label = { t }, x label style={ yshift = -15pt }, y label style={rotate=270, yshift=340pt}, class placement transform={scale=3.3}
		]
		\class["1"](0,0)
		\class["u"](-2,0)
		\class["u^2"](-4,0)
		\class["u^3"](-6,0)
		\class["u^4"](-8,0)
		\class["u^5"](-10,0)
		\class["\tau_0"](0,1)
		\class["u\tau_0"](-2,1)
		\class["u^2\tau_0"](-4,1)
		\class["u^3\tau_0"](-6,1)
		\class["u^4\tau_0"](-8,1)
		\class["u^5\tau_0"](-10,1)
		\class["\xi_1"](0,4)
		\class["u\xi_1"](-2,4)
		\class["u^2\xi_1"](-4,4)
		\class["u^3\xi_1"](-6,4)
		\class["u^4\xi_1"](-8,4)
		\class["u^5\xi_1"](-10,4)
		\class["\xi_1\tau_0"](0,5)
		\class["u\xi_1\tau_0"](-2,5)
		\class["u^2\xi_1\tau_0"](-4,5)
		\class["u^3\xi_1\tau_0"](-6,5)
		\class["u^4\xi_1\tau_0"](-8,5)
		\class["u^5\xi_1\tau_0"](-10,5)
		\class["\tau_1"](0,5)
		\class["u\tau_1"](-2,5)
		\class["u^2\tau_1"](-4,5)
		\class["u^3\tau_1"](-6,5)
		\class["u^4\tau_1"](-8,5)
		\class["u^5\tau_1"](-10,5)
		\d2(-2,0)(0,1)
		\d2(-4,0)(-2,1)
		\d2(-8,0)(-6,1)
		\d2(-10,0)(-8,1)
		\d2(-2,4)(0,5)
		\d2(-4,4)(-2,5)
		\d2(-8,4)(-6,5)
		\d2(-10,4)(-8,5)
	\end{sseqpage}
\end{center}
\begin{example}\label{example3.4}  We can now compare with the BAHSS computing $\THH(\tmf,\fff_2)$, which has signature $$E_1^{s,t}=\THH_{-s}(\fff_2,\pi_t(\fff_2\otimes_{\tmf}\fff_2))\implies \THH_{t-s}(\tmf,\fff_2).$$   Using \cite[~Theorem 5.13]{mathew2015homology}, either that $\tmf\otimes_{\sss}\fff_2$ forms a sub-Hopf-algebra of $\fff_2\otimes_{\sss}\fff_2$ or explicitly computing the Hopf conjugates of the $\zeta_i$ using \cite{198669}, one finds that $$H_*(\tmf,\fff_2)=\fff_2[\xi_1^8,\xi_2^4,\xi_3^2,\xi_4,\ldots]\subseteq \pi_*(\fff_2\otimes_{\sss}\fff_2).$$  Using that $\fff_2\otimes_{\tmf}\fff_2\simeq \fff_2\otimes_{\tmf\otimes_{\sss}{\fff_2}}(\fff_2\otimes_{\sss}\fff_2)$, and since everything is nice and polynomial, we get that $$\pi_*(\fff_2\otimes_{\tmf}\fff_2)=\fff_2[\xi_1,\xi_2,\xi_3]/(\xi_1^8,\xi_2^4,\xi_3^2)$$ as a quotient of the dual Steenrod algebra.  Thus, our $E_1$-page takes the form
$$E_1^{s,t}=\fff_2[u]\otimes_{\fff_2}\fff_2[\xi_1,\xi_2,\xi_3]/(\xi_1^8,\xi_2^4,\xi_3^2).$$  
Using the comparison map from the BAHSS considered in Example \ref{ex3.1}, we find that $d_1(u)=\xi_1$, leaving 
$$E_2^{s,t}\simeq \fff_2[u^2]\otimes_{\fff_2}\fff_2[\xi_2,\xi_3]/(\xi_2^4,\xi_3^2)\otimes_{\fff_2}\Lambda[\xi_1^7u].$$
 Similarly, $u^2$ maps to $\xi_2$ on the $E_3$-page, leaving
  $$E_4^{s,t}\simeq \fff_2[u^4]\otimes_{\fff_2}\fff_2[\xi_3]/(\xi_3^2)\otimes_{\fff_2}\Lambda[\xi_1^7u, \xi_2^3 u^2],$$ 
  and finally $u^4$ maps to $\xi_3$ on the $E_7$ page, giving $$E_8^{s,t}=E_{\infty}^{s,t}=\fff_2[u^8]\otimes_{\fff_2}\Lambda(u\xi_1^7, u^2 \xi_2^3, u^4\xi_3).$$ 
  At this point, the spectral sequence must degenerate, as $u^8$ cannot hit any nonzero class.  There are no multiplicative extension problems for degree reasons, and we find that 
  $$\pi_*\THH(\tmf,\fff_2)\simeq \pi_*\THH(\fff_2,\fff_2\otimes_{\tmf}\fff_2)\simeq \fff_2[u^8]\otimes_{\fff_2}\Lambda(u\xi_1^7,u^2\xi_2^3, u^4\xi_3)$$
   is the tensor product of a polynomial algebra with a generator in degree 16 with an exterior algebra on classes in degrees 9, 13, and 15.  For future reference, let us write $$\pi_*(\THH(\tmf,\fff_2))=\fff_2[\beta]\otimes_{\fff_2}\Lambda(\lambda_1,\lambda_2,\lambda_3),$$ with $|\beta|=16$, $|\lambda_1|=9$, $|\lambda_2|=13$, and $|\lambda_3|=15$.  We remark that is also a direct consequence of \cite[~Theorem 6.2(b)]{angeltveit2005hopf}.
\end{example}

Before we compute $\THH(\tmf,\zz_{(2)})$, let's warm up by reproving Angelveit-Hill-Lawson's \cite{angeltveit2009topological} computation of $\THH(\ell,\zz_{(p)})$ at an odd prime $p$.
\begin{example}\label{ex3.5}  In a fashion similar to Example \ref{example3.4}, using \cite[~\textsection2]{1986119} and the identification of the Adams summand $\ell$ with the truncated Johnson-Wilson spectrum $\mo{BP}<1>$ (the equivalence is described in the same section), one gets that $\pi_*(\fff_p\otimes_{\ell}\fff_p)\simeq \Lambda[\tau_0,\tau_1]$.  With this in hand, one computes
	$$\THH_*(\ell,\fff_{p})\simeq \fff_p[u^{p^2}]\otimes_{\fff_p}\Lambda[u^{p-1}\tau_0,u^{p(p-1)}\tau_1].$$ 
	 Since $\pi_*(\ell)\simeq \zz_{(p)}[v_1]$, with $|v_1|=2(p-1)$, it is easy to see that $\pi_*(\zz_{(p)}\otimes_{\ell}\zz_{(p)})\simeq \Lambda_{\zz_{(p)}}[\rho]$, with $|\rho|=2p-1$.  Thus, the $E_1$-page of our Brun-Atiyah-Hirzebruch spectral sequence computing $\THH(\ell,\zz_{(p)})$ is given by $$E_{1}=\THH_*(\zz_{(p)})\otimes_{\zz_{(p)}} \Lambda[\rho],$$ with $\rho$ in bidegree $(0,2p-1)$, and $a\in\THH_i(\zz_{(p)})$ in bidegree $(-i,0)$. 
	 
The $\fff_p$-vector space $\pi_{2p-1}(\THH(\ell,\fff_p))$ is one dimensional, which implies that there must be a multiplicative extension between $\rho$ and the $\zz/p\zz$ class in $\THH_{2p-1}(\zz_{(p)})$.  By examining the locations where elements of $\THH(\ell,\fff_p)$ are nonzero, and working inductively, we find that for all $r>1$, there is a nonzero differential on the class $\THH_{2rp-1}(\zz_{(p)})$, which ends up giving us that 
 $$\THH_*(\ell,\zz_{(p)})=\begin{cases}
 	\zz_{(p)}  &\qquad \text{if } *=0, 2p-1\\
 	\zz/p^k\zz &\qquad \text{if } *=2rp^{k+1}-1, 2rp^{k+1}-1+2p-1 \text{ with } k>0, \gcd(r,k)=1\\
 	0 & \qquad \text{otherwise}.
 \end{cases}$$
\end{example}

Now for the more complicated case of $\THH(\tmf,\zz_{(2)})$.  First, we need to understand $\pi_*(\zz_{(2)}\otimes_{tmf}\zz_{(2)})$.  To get at this, we will use two Bockstein spectral sequences, and compare with the following preliminary computations.
\begin{example}\label{example3.6}
First, we examine the Bockstein spectral sequence coming from $\zz_{(2)}^{bok}\otimes_{\sss}\fff_2$ computing $\pi_*(\zz_{(2)}\otimes_{\sss}\fff_2)$, with $E_1$-page $\fff_2[\xi_1,\xi_2,\ldots][\tilde{v_0}]$, with $(s,t)$-degree $|\xi_i|=(1-2^i,0)$ and $|\tilde{v_0}|=(1,1)$.  Recalling either computations from \cite[~Theorem 3.2]{Kochman_1981}, or using the fact that since the sphere is connective, $\pi_0(\zz_{(2)}\otimes_{\sss}\fff_2)\simeq \zz_{(2)}\otimes \fff_2\simeq \fff_2$, we find that the class $\tilde{v_0}$ must vanish in this spectral sequence, and the only way for that to happen is for $d_1(\xi_1)=\tilde{v_0}$.  The Bockstein spectral sequence then collapses at page 2 since there are no nonzero classes on $E_2^{s,t}$ with positive $t$-degree.  This recovers the classic computation $\pi_*(\zz_{(2)}\otimes_{\sss}\fff_2)\simeq \fff_2[\xi_1^2,\xi_2,\xi_3,\ldots]$.
\end{example}
For the second computation, we examine the Bockstein spectral sequence attached to $\zz_{(2)}\otimes_{\sss}\zz_{(2)}^{bok}$, with $E_1$-page $\fff_2[\xi_1^2,\xi_2,\ldots][\tilde{v_0}]$.
\begin{lemma}\label{lem3.7}
In the Bockstein spectral sequence attached to $\zz_{(2)}\otimes_{\sss}\zz_{(2)}^{bok}$, we have $d_1(\xi_{i})=\xi_{i-1}^2\tilde{v_0}$ for all $i>1$, and the spectral sequence collapses on page 2.
\end{lemma}
\begin{proof}
This follows from \cite[~Lemma 3.3b]{Kochman_1981}, where Kochman proves that the Bockstein $\beta$ on $\fff_2[\xi_1^2,\xi_2,\ldots][\tilde{v_0}]$ acts as $\beta(\xi_1^2)=0$ and $\beta(\xi_i)=\xi_{i-1}^2$ for $i>1$.
\end{proof}
\vskip1pt
\begin{center}
	\begin{sseqpage}[title={part of the $E_1$-page in low degrees}, axes type = center,  classes = {draw = none },
		x range = {-5}{1}, y range = {0}{2}, x axis origin = {1}, y tick gap = {-0.5 cm}, x label = { s }, y label = { t }, x label style={ yshift = -15pt }, y label style={rotate=270, yshift=490pt, xshift=23pt}, class placement transform={scale=1.55}, xscale=2
		]
		\class["1"](0,0)
		\class["\xi_1^2"](-2,0)
		\class["\xi_2"](-3,0)
		\class["\xi_1^4"](-4,0)
		\class["\xi_1^2\xi_2"](-5,0)
		\class["\tilde{v}_0\xi_1^2"](-1,1)
		\class["\tilde{v}_0\xi_2"](-2,1)
		\class["\tilde{v}_0\xi_1^4"](-3,1)
		\class["\tilde{v}_0\xi_1^2\xi_2"](-4,1)
		\class["\tilde{v}_0^2\xi_1^2"](0,2)
		\class["\tilde{v}_0^2\xi_2"](-1,2)
		\class["\tilde{v}_0^2\xi_1^4"](-2,2)
		\class["\tilde{v}_0^2\xi_1^2\xi_2"](-3,2)
		\d1(-3,0)(-1,1)
		\d1(-2,1)(0,2)
		\d1(-5,0)(-3,1)
		\d1(-4,1)(-2,2)
	\end{sseqpage}
\end{center}
\begin{remark}\label{remark3.8}
Lemma 3.6 allows us to illuminate the structure of the integral dual Steenrod algebra $(\mc{A}_{\zz})_*:=\pi_*(\zz_{(2)}\otimes_{\sss}\zz_{(2)})$, using the fact that there is no room for multiplicative extensions (beyond the trivial $\tilde{v}_0=2$ one) in the Bockstein spectral sequence we have just computed.  First and foremost, we find there is a sub-algebra $R:=\zz_{(2)}[\xi_1^2,\xi_2^2,\ldots]/(2\xi_1^2,2\xi_2^2,\ldots)$, which is almost polynomial on the $\xi_i^2$, except that we enforce they are all $0$ after multiplication by 2.  Since the Bockstein spectral sequence collapses on page 2, given any class $a\in E_1^{-*,0}$ ($*>0$) with $d_1(a)=0$, there must exist some class $b\in E_1^{-*-1,0}$ with $d_1(b)=a\tilde{v}_0$.  Rephrasing this, every class $a\in (\mc{A}_{\zz})_*$ in positive degree arises as $a=d_1(b)/\tilde{v}_0$ for some $b\in \fff_2[\xi_1^2,\xi_2,\xi_3,\ldots]$.  Since $d_1(\xi_j^2b)=\xi_j^2d_1(b)$ for any $b\in\fff_2[\xi_1^2,\xi_2,\xi_3,\ldots]$, we can generate $(\mc{A}_{\zz})_*$ as an $R$-module using the classes $d(\xi_{i_1}\ldots\xi_{i_k})$ for $1<i_1<\ldots <i_k$, as these span $\fff_2[\xi_1^2,\xi_2,\xi_3,\ldots]$ as an $\fff_2[\xi_1^2,\xi_2^2,\xi_3^2,\ldots]$-module.  We will call $d_1(\xi_{i_1}\ldots \xi_{i_k})/\tilde{v_0}$ the ``torsion product'' of $\xi_{i_1-1}^2,\ldots, \xi_{i_k-1}^2$, and sometimes alternatively denote it $\xi_{i_1-1}^2*\ldots*\xi_{i_k-1}^2$.  With this in mind, we remark that $(\mc{A}_{\zz})_*$ maps multiplicatively to $\fff_2[\xi_1^2,\xi_2,\xi_3,\ldots]$ by taking everything modulo 2, and this map is injective in positive degrees, so were it not for the 2-torsion, $(\mc{A}_{\zz})_*$ would be an integral domain.
\end{remark}

Now we can finally compute the homotopy groups (and multiplicative structure on them) of the integral dual Steenrod algebra over $\tmf$.
\begin{proposition}\label{prop3.9}
	The homotopy groups of $\zz_{(2)}\otimes_{\tmf}\zz_{(2)}$ are given by
$$\pi_*(\zz_{(2)}\otimes_{\tmf}\zz_{(2)})\simeq \begin{cases}
	\zz_{(2)} &\qquad \text{if } *=0,22,\\
	\zz_{(2)}\oplus\zz/2\zz &\qquad \text{if } *=9,13,\\
	\zz/2\zz & \qquad \text{if } *=2,4,8,10,11,12,17,19,\\
	\zz/2\zz\oplus\zz/2\zz &\qquad \text{if } *=6,15\\
	0 &\qquad \text{otherwise}.
\end{cases}$$
As a graded ring, we have that $$\pi_*(\zz_{(2)}\otimes_{\tmf}\zz_{(2)})\simeq \zz_{(2)}[x,y,u,v,w]/(x^4,y^2,w^2,2w,2x,2y,u^2,v^2,yv,xu,uw,vw,x^3w-uy,yw-xv),$$ with $|x|=2$, $|y|=6$, $|u|=9$, $|v|=13$, and $|w|=9$
\end{proposition}
\begin{proof}
We begin as in Example \ref{example3.6} with the Bockstein spectral sequence attached to $\zz_{(2)}^{bok}\otimes_{\tmf}\fff_2$.  This has $E_1$-page $\fff_2[\xi_1,\xi_2,\xi_3]/(\xi_1^8,\xi_2^4,\xi_3^2)[\tilde{v_0}]$, and as in Example \ref{example3.6}, has $d_1(\xi_1)=\tilde{v_0}$ and collapses on page 2.  

Turning attention now to the Bockstein spectral sequence attached to $\zz_{(2)}\otimes_{\tmf}\zz_{(2)}^{bok}$, the $E_1$-page of this spectral sequence has the form $\fff_2[\xi_1^2,\xi_2,\xi_3]/(\xi_1^{8},\xi_2^4,\xi_3^2)[\tilde{v_0}]$.  The map of filtered objects $\zz_{(2)}\otimes_{\sss}\zz_{(2)}^{bok} \to \zz_{(2)}\otimes_{\tmf} \zz_{(2)}^{bok}$ extends to a map on spectral sequences, taking $\xi_i$ to $\xi_i$ and $\tilde{v_0}$ to $\tilde{v_0}$.  Comparing with Lemma \ref{lem3.7} tells us that $d_1(\tilde{v_0})=d_1(\xi_1^2)=0$, $d_1(\xi_2)=\xi_1^2\tilde{v_0}$, $d_1(\xi_3)=\xi_2^2\tilde{v_0}$, and all $d_1$ differentials are generated by multiplicativity from these.

The $E_2$-page is now given by $$E_2=\fff_2[\tilde{v_0},\xi_1^2,\xi_2^2,u,v,w]/(\xi_1^8,\xi_2^4,w^2,\tilde{v_0}\xi_1^2,\tilde{v_0}\xi_2^2,\tilde{v_0}w,u^2,v^2,\xi_2^2v,\xi_1^2u,uw,vw,\xi_1^6w-u\xi_2^2,\xi_2^2w-\xi_1^2v),$$ with $|u|=(-9,0)$ ($u$ represents the class ``$\xi_2\xi_1^6$''), $|v|=(-13,0)$ ($v$ represents the class ``$\xi_3\xi_2^2$''), and $|w|=(-9,0)$ represents the class $\xi_3\xi_1^2+\xi_2^3$.  With these conventions, we can draw out the $E_2$-page.
\begin{center}
	\begin{sseqpage}[title={$E_2$-page of the Bockstein Spectral Sequence}, axes type = center,  classes = {draw = none },
		x range = {-22}{1}, y range = {0}{3}, x axis origin = {1}, y tick gap = {-0.2 cm}, x label = { s }, y label = { t }, x label style={ yshift = -15pt }, y label style={rotate=270, yshift=460pt, xshift=23pt}, class placement transform={scale=2}, xscale=0.67
		]
		\class["1"](0,0)
		\class["\tilde{v}_0"](1,1)
		\class["\xi_1^2"](-2,0)
		\class["\xi_1^4"](-4,0)
		\class["\xi_2^2"](-6,0)
		\class["\xi_1^6"](-6,0)
		\class["u"](-9,0)
		\class["v"](-13,0)
		\class["\xi_2^2u"](-15,0)
		\class["\xi_1^2v"](-15,0)
		\class["\xi_1^4v"](-17,0)
		\class["\xi_1^6v"](-19,0)
		\class["uv"](-22,0)
		\class["\tilde{v}_0u"](-8,1)
		\class["\tilde{v}_0v"](-12,1)
		\class["\tilde{v}_0^2u"](-7,2)
		\class["\tilde{v}_0^2v"](-11,2)
		\class["\tilde{v}_0^3u"](-6,3)
		\class["\tilde{v}_0^3v"](-10,3)
		\class["\tilde{v}_0uv"](-21,1)
		\class["\tilde{v}_0^2uv"](-20,2)
		\class["\tilde{v}_0^3uv"](-19,3)
	\end{sseqpage}
\end{center}  For the rest of the spectral sequence, the only classes that can support a nonzero differential are $u,v$ and $uv$, which would have to hit a class in degrees $(n-8,n)$, $(n-12,n)$ or $(n-21,n)$, respectively, of which there are none.  Thus, the spectral sequence collapses on $E_2$.  Working out the extension problems, $u,v$ and $uv$ give a $\zz_{(2)}$-tower in degrees $9, 13$ and $22$, respectively (using Proposition \ref{proposition2.11}), and the rest of the classes contribute the listed $\zz/2\zz$ summands.  The multiplicative structure follows from multiplicativity of the spectral sequence, and since the only possible extension problems arise from the $\tilde{v_0}$-classes, which account for multiplication by 2.  This gives the desired presentation with $x:=\xi_1^2$, $y:=\xi_2^2$, and $u,v,w$ as defined along the way.
\end{proof}

This allows us to prove a preliminary lemma.
\begin{lemma}\label{lem3.10}
The topological Hochschild homology of $\tmf$ with rational coefficients is given by $$\THH_*(\tmf,\bb{Q})=\begin{cases}
	\bb{Q} \qquad\qquad & \text{ if }*=0,9,13,22\\
	0 \qquad\qquad &\text{ otherwise.}
\end{cases}$$
\end{lemma}
\begin{proof}
One notes that $\THH(\tmf,\bb{Q})=\THH(\tmf,\zz_{(2)})[\f{1}{2}]$, so it suffices to examine $\THH(\tmf,\zz_{(2)})$ up to torsion.  Setting up the Brun-Atiyah-Hirzebruch spectral sequence for $\THH(\tmf,\zz_{(2)})$, Proposition \ref{prop3.9} shows that the $E_1$-page is given by $\THH_{*}(\zz_{(2)})\oplus\text{ torsion}$ in bidegrees $(-*,0),(-*,9),(-*,13),(-*,22)$, with all other entries $2$-torsion and with only finitely many finitely generated abelian groups along any given $(t-s)$-line.  This implies that the only possible torsion-free summands of homotopy groups of $\THH(\tmf,\zz_{(2)})$ are the copies of $\zz_{(2)}$ in degrees $0,9,13,22$, which persist throguh the Brun-Atiyah-Hirzebruch spectral sequence because every class that could potentially hit them on any given page is $2$-power torsion, and these classes cannot support nonzero differentials.  Localizing at 2 (which kills off the torsion), we get the claim.
\end{proof}
Proving the main theorem of this section will require a careful analysis of the Brun-Atiyah-Hirzebruch spectral sequence for $\THH(\tmf,\zz_{(2)})$.  To get some preliminary information, one can examine a variant of this, as follows.
\begin{example}\label{ex3.11} For $R$ an $\bb{E}_{\infty}$-ring spectrum with a map to $\zz_{(2)}$, we wish to examine the Brun-Atiyah-Hirzebruch spectral sequence computing $\THH(R,\zz_{(2)})$.  The filtered object \\$\THH(R,\tau_{\geq *}(\zz_{(2)}\otimes_{R} \zz_{(2)}))$ admits a map to the filtered object $\THH(\zz_{(2)},\tau_{\geq *}(\zz_{(2)}\otimes_{R}\fff_2))$, whose associated spectral sequence converges to $\THH(R,\fff_2)$, we will call this the mod $2$ Brun-Atiyah-Hirzebruch spectral sequence and denote it $\mo{BAHSS}_{\zz_{(2)}}^{R}/2$.  Using the comparison, by first analyzing $\mo{BAHSS}_{\zz_{(2)}}^{R}/2$, we can gain some insight into behavior of the full BAHSS.  

As usual, we begin with the case $R=\sss$.  From Example \ref{example3.6}, we have that $\pi_*(\zz_{(2)}\otimes_{R}\fff_2)=\fff_2[\xi_1^2,\xi_2,\xi_3,\ldots]$ with $|\xi_1|^2=2$ and $|\xi_i|=2^{i}-1$ for $i>1$.  We write $\THH_*(\zz_{(2)},\fff_{2})=\fff_2[\alpha]\otimes \Lambda_{\fff_2}(\gamma)$ with $|\gamma|=3$ and $|\alpha|=4$.  Since the spectral sequence must converge to $\fff_2$, one finds similar to Example \ref{ex3.1} that we must have $d_2(\gamma)=\xi_1^2$, and $d_{2^{i-1}}(\alpha^{2^{i-2}})=\xi_{i}$ for $i\geq 2$, with the rest of the differentials determined by multiplicativity.  Examining now the case $R=\tmf$, we have that $\pi_*(\zz_{(2)}\otimes_{\tmf}\fff_2)=\fff_2(\xi_1^2,\xi_2,\xi_3)/(\xi_1^8,\xi_2^4,\xi_3^2)$, and once again, for degree reasons, all differentials are determined by $d_2(\gamma)=\xi_1^2$, $d_3(\alpha)=\xi_2$ and $d_7(\alpha^2)=\xi_3$.
\end{example}
We compute one final input before the main theorem.
\begin{example}\label{ex3.12} Consider the BAHSS computing $\THH_*(\sss,\zz_{(2)})\simeq\zz_{(2)}$.  Using Remark \ref{remark3.8}, we get the $E_1$-page of the spectral sequence.  Recall from \cite{bokstedt1993topological} that $$\THH_*(\zz_{(2)})=\begin{cases}
		\zz_{(2)} & \qquad \text{if }*=0\\
		\zz/2^{i}\zz & \qquad \text{if }*=2^{i+1}r-1,\text{ for }r \text{ odd}\\
		0 & \qquad \text{otherwise.}
\end{cases}$$  For $i\geq 1$, we will use $\beta_i$ to denote some chosen generator of $\THH_{4i-1}(\zz_{(2)})$ as a cyclic group.  With notation as before, we will write $\THH_*(\zz_{(2)},\fff_2)=\fff_2[\alpha]\otimes_{\fff_2}\Lambda_{\fff_2}(\gamma)$, and will abuse notation as if these represented actual classes on the $E_1$-page, e.g., writing the nonzero class in $E^{-4,2}_1$ as $\xi_1^2\alpha$ (even though this class is not literally a product of $\xi_1^2$ with any class ``$\alpha$'').  On the other hand, the classes $\xi_1^2\alpha^{i-1}\gamma$ are the product $\xi_1^2\cdot \beta_i$ for $i\geq 1$, so these particular classes do arise as products on the $E_1$-page.

Let's introduce a convenient notation.  By Remark \ref{remark3.8}, any positive degree element $x\in (\mc{A}_{\zz})_*$ can be written in the form ``$x=d_1(y)$'' with $y\in \fff_2[\xi_1^2,\xi_2,\ldots]$, and $d_1$ the $\fff_2$-derivation of this algebra uniquely determined by specifying $d_1(\xi_1^2)=0$, $d_1(\xi_{j})=\xi_{j-1}^2=0$.  Since we are computing spectral sequences here, the notation ``$d_1$'' is rather ambiguous, so we will henceforth use the notation ``$d^S$'' for the derivation we have just defined.

The key technique is to use the comparison with the mod 2 BAHSS computed in the first half of Example \ref{ex3.11}.  This map induces an injection everywhere on the $E_1$-page, except on the line $E_1^{-*,0}$, where it is given by the quotient map $\THH_{-*}(\zz_{(2)})\to \THH_{-*}(\zz_{(2)},\fff_2)$.  In positive $t$-degree on the $E_1$-page, the image is generated as a module over $\fff_2[\xi_1^2,\xi_2^2,\ldots][\alpha,\gamma]$ by torsion products $\xi_{i_1}^2*\ldots \xi_{i_k}^2$ for $k\geq 1$, $1\leq i_1<i_2<\ldots <i_k$, as defined in Remark \ref{remark3.8}.  We will first analyze through the $E_6$-page to illuminate the method, then generalize to describe the rest of the BAHSS.

To begin, this comparison map tells us that there are no $d_1$-differentials, and that the $d_2$-differentials are determined by $d_2(\beta_i)=\alpha^{i-1}\xi_1^2$ for $i\geq 1$, together with multiplicativity (explicitly, given a class $x\in (\mc{A}_{\zz})_*$ in positive degree, we have that $d_2(\alpha^i\gamma x)=\alpha^i\xi_1^2x$.  This kills off every class divisible by $\xi_1^2$, and on $E_3$-pages, we still have an injection $$E_3^{*,t}(\mo{BAHSS}_{\zz_{(2)}})\hookrightarrow E_3^{*,t}(\mo{BAHSS}_{\zz_{(2)}}/2)$$ whenever $t>0$, with image generated as a module over $\fff_2[\xi_2^2,\xi_3^2,\ldots][\alpha]$ by most of the same torsion products as before, $\xi_{i_1}^2*\ldots \xi_{i_k}^2$ for $k\geq 1$, $1\leq i_1<i_2<\ldots <i_k$- where we enforce $\xi_1^2=0$.  Next, on the $E_3$-page, in the mod 2 BAHSS, we had $d_3(\alpha)=\xi_2$.  These differentials lift to everything in positive degree, and we claim that $E_4^{*,t}(\mo{BAHSS}_{\zz_{(2)}}^{\sss})$ injects into $E_4^{*,t}(\mo{BAHSS}_{\zz_{(2)}}^{\sss}/2)$ for all $t>0$, except when $t=6$.  For $t=6$, we have $E_4^{*,6}(\mo{BAHSS}_{\zz}^{\sss}/2)=0$, yet $\alpha^{2i}\xi_2^2\neq 0$ in $E_4^{*,6}(\mo{BAHSS}_{\zz_{(2)}}^{\sss})$ for all $i\geq 0$.  The latter claim holds since the classes $\alpha^{2i-1}\xi_2$ do not exist in our BAHSS to hit these classes as they do modulo 2.

For the former claim, in fact one can say more, that $E_4^{*,t}$ for $t>6$ is generated as an $\fff_2[\xi_3^2,\xi_4^2,\ldots][\alpha^2]$-module by torsion products $\xi_{i_1}^2*\ldots \xi_{i_k}^2$ for $k\geq 1$, with $2\leq i_1<i_2<\ldots <i_k$ (where we enforce ``$\xi_2^2=0$'' when this class arises in expressions).  The terms from the $E_3^{0,*}$-line (with $*>6$) which become boundaries in $E_3(\mo{BAHSS}_{\zz_{(2)}}^{\sss}/2)$ are precisely the terms which modulo 2 are divisible by $\xi_2$.  These are precisely those $x\in (\mc{A}_{\zz})_*$ which can be written as $x=d^S(y)$ for some $y\in \fff_2[\xi_2,\xi_3,\ldots]$, and such that $\xi_2|x$.  Since we have that $\xi_1^2=0$ in $E_3(\mo{BAHSS}_{\zz_{(2)}}^{\sss}/2)$, our derivation $d^S$ satisfies $d^S(\xi_2)=0$ on the $E_3$-page, so that $d^S(\xi_2 y)=\xi_2d^S(y)$ for any $y\in \fff_2[\xi_2,\xi_3,\ldots]$.  We need to see that for each class $x\in E_3^{0,>6}(\mo{BAHSS}_{\zz_{(2)}}^{\sss})$ becoming a boundary in $E_3(\mo{BAHSS}_{\zz_{(2)}}^{\sss}/2)$, the differential hitting $x$ lifts to $E_3(\mo{BAHSS}_{\zz_{(2)}}^{\sss})$.  For this, note that if $\xi_2|x$, then $d^S((x/\xi_2)\xi_2)=\xi_2d^S(x/\xi_2)=0\mod(\xi_1^2)$, which forces $d^S(x/\xi_2)=0$.  On the quotient algebra $\fff_2[\xi_2,\xi_3,\ldots]$, $d^S$ almost surjects onto it's kernel, with $\mo{Ker}(d^S)/\mo{Im}(d^S)$ a 2-dimensional space generated by $1, \xi_2$.  As $x$ was homogeneous of degree greater than $6$, $x/\xi_2$ is homogeneous of degree greater than $3$, so must lie in the image of $d^S$, giving us some $y\in \fff_2[\xi_2,\xi_3,\ldots]$ with $d^S(y)=x/\xi_2$.  But then the class $x^{\prime}=d^S(y)$ lives in $(\mc{A}_{\zz})_*$, and modulo $(\xi_1^2)$, $\xi_2x^{\prime}=x$, such that $\alpha x^{\prime}$ maps to the class of the same name in $E_3(\mo{BAHSS}_{\zz_{(2)}}^{\sss}/2)$ with $d_3(\alpha x^{\prime})=x$, and we must have $d_3(\alpha x^{\prime})=x$ in $E_3(\mo{BAHSS}_{\zz_{(2)}}^{\sss})$ as well.  This shows that given any class in $E_3^{*,>6}(\mo{BAHSS}_{\zz_{(2)}}^{\sss})$ which becomes a boundary in $E_3(\mo{BAHSS}_{\zz_{(2)}}^{\sss})$, we can find a class hitting it under the $d_3$-differential, which shows the claimed injectivity on $E_4$.  Since every nontrivial class in $E_{4}^{*,>6}(\mo{BAHSS}_{\zz_{(2)}}^{\sss})$ maps to a nontrivial class in $E_4(\mo{BAHSS}_{\zz_{(2)}}^{\sss}/2)$, and every class in $E_4^{*,\leq 6}(\mo{BAHSS}_{\zz_{(2)}}^{\sss})$ maps to 0 in the mod 2 BAHSS, there can be no nontrivial differentials from the $E_4^{*,\leq 6}$ classes to a class in $E_4^{*,>6}$ until at least $E_7$.  The $E_3$-page we just described is pictured below (in low degrees).
\begin{center}
	\begin{sseqpage}[title={$E_3(\mo{BAHSS}_{\zz_{(2)}}^{\sss})$}, axes type = center,  classes = {draw = none },
		x range = {-16}{0}, y range = {0}{12}, x axis origin = {1}, y tick gap = {-0.5 cm}, x label = { s }, y label = { t }, x label style={ yshift = -20pt }, y label style={rotate=270, yshift=360pt, xshift=23pt}, class placement transform={scale=2.9}, xscale=0.7, yscale=0.5
		]
		\class["\zz_{(2)}"](0,0)
		\class["\zz/2\zz \beta_2"](-7,0)
		\class["\zz/4\zz\beta_4"](-15,0)
		\class["\xi_2^2"](0,6)
		\class["\alpha\xi_2^2"](-4,6)
		\class["\alpha^2\xi_2^2"](-8,6)
		\class["\alpha^3\xi_2^2"](-12,6)
		\class["\alpha^4\xi_2^2"](-16,6)
		\class["\xi_2^4"](0,12)
		\class["\alpha\xi_2^4"](-4,12)
		\class["\alpha^2\xi_2^4"](-8,12)
		\class["\alpha^3\xi_2^4"](-12,12)
		\class["\alpha^4\xi_2^4"](-16,12)
		\class["\xi_1^2*\xi_2^4"](0,9)
		\class["\alpha\xi_1^2*\xi_2^2"](-4,9)
		\class["\alpha^2\xi_1^2*\xi_2^2"](-8,9)
		\class["\alpha^3\xi_1^2*\xi_2^2"](-12,9)
		\class["\alpha^4\xi_1^2*\xi_2^2"](-16,9)
		\d3(-4,6)(0,9)
		\d3(-12,6)(-8,9)
		\d3(-4,9)(0,12)
		\d3(-12,9)(-8,12)
	\end{sseqpage}
\end{center}

We now describe what happens in the general case.
\end{example}
\begin{theorem}\label{th3.13}
Consider the Brun-Atiyah-Hirzebruch spectral sequence for the sphere with $E_1$-page $E_1^{s,t}(\mo{BAHSS}_{\zz}^{\sss})=\THH_{-s}(\zz_{(2)},\pi_{t}(\zz_{(2)}\otimes_{\sss}\zz_{(2)}))$.  Then,
\begin{enumerate}
\item  There are no $d_r$-differentials unless $r=2^{i}-1$ or $2^{i}-2$ for some $i\geq 2$, and the $d_{2^{i}-2}$-differential is 0 on any element with positive $t$-degree, whenever $i>2$.
\item The portion $E_{2^{i}-1}^{*,>0}(\mo{BAHSS}_{\zz_{(2)}}^{\sss})$ of the $E_{2^{i}-1}$-page injects into $E_{2^{i}-1}^{*,>0}(\mo{BAHSS}_{\zz_{(2)}}^{\sss}/2)$ with image given by the $\fff_2[\xi_i^2,\xi_{i+1}^2,\ldots][\alpha^{2^{i-2}}]$-submodule of the target generated by the torsion products $\xi_{i_1}^2*\ldots \xi_{i_k}^2$ for $k\geq 1$, $i-1\leq i_1<i_2<\ldots <i_k$ (where we enforce ``$\xi_{i-1}^2=0$'' when taking these torsion products).
\item The $d_{2^{i}-1}$-differential of $\mo{BAHSS}_{\zz_{(2)}}^{\sss}$ is completely determined by this injection, and the stipulation that $d_{2^{i}-1}=0$ on the classes with $t$-degree $0$.
\item  The $E_{2^{i}-2}^{*,>0}(\mo{BAHSS}_{\zz_{(2)}}^{\sss})$-portion of the $E_{2^i-2}$-page differs from $E_{2^{i}-1}^{*,>0}(\mo{BAHSS}_{\zz_{(2)}}^{\sss})$ only in the line $t=2^{i}-2$, where $E_{2^{i}-2}^{*,>0}(\mo{BAHSS}_{\zz_{(2)}}^{\sss})$ has non-zero classes given by $\alpha^{k2^{i-1}}\xi_{i-1}^2$ for $k\geq 0$.
\item  The $d_{2^{i}-2}$-differential is determined by $d_{2^{i}-2}(2^{i-2}\beta_{k})=\alpha^{k2^{i-2}}\xi_{i-1}^2$ for $k\geq 0$, with $d_{2^{i}-2}$ vanishing on every other class (for $i\geq 3$).
\end{enumerate}   
\end{theorem}
\begin{proof}
By induction, assume we know that for some $i\geq 2$, the map $$E_{2^{i}}^{*,t}(\mo{BAHSS}_{\zz_{(2)}}^{\sss})\to E_{2^{i}}^{*,t}(\mo{BAHSS}_{\zz_{(2)}}^{\sss}/2)$$ is the zero map for $t\leq 2^{i+1}-2$, and is an injection for $t>2^{i+1}-2$, with image generated as an $\fff_2[\xi_{i+1}^2, \xi_{i+2}^2,\ldots][\alpha^{2^{i-1}}]$-module by torsion products (mod $\xi_1^2,\xi_2,\ldots,\xi_i$) $\xi_{i_1}^2*\ldots \xi_{i_k}^2$ for $k\geq 1$, $i\leq i_1<i_2<\ldots <i_k$.  Then since $E_{k}(\mo{BAHSS}_{\zz_{(2)}}^{\sss}/2)$ supports no nontrivial differentials for $2^i\leq k<2^{i+1}-1$, the same claim holds on the $E_{2^{i+1}-1}$-page.   

Now, suppose that we have a class $x\in E_{2^{i+1}-1}^{*,>2^{i+1}-2}(\mo{BAHSS}_{\zz_{(2)}}^{\sss})$ which becomes a boundary in $\mo{BAHSS}_{\zz_{(2)}}^{\sss}/2$.  This forces us to have $x=\alpha^{k2^{i}}\cdot \xi_{i+1}\cdot x^{\prime}\mod(\xi_1^2,\ldots, \xi_{i})$ with $k\geq 0$.  We can pull out the $\alpha^{k2^i}$ to assume that $x\in E_{2^{i+1}-1}^{0,*}$ is a class arising from $(\mc{A}_{\zz})_*$, which modulo $2,\xi_1^2,\ldots,\xi_{i}$, is divisible by $\xi_{i+1}$.  Write $x=d^S(y)$ for some $y\in \fff_2[\xi_{i+1},\xi_{i+2},\ldots]$, where as $d^S(\xi_{i+1})=0 \mod (\xi_1^2,\ldots,\xi_{i})$, we have $0=d^S(\xi_{i+1}\cdot x/\xi_{i+1})=\xi_{i+1}d^S(x/\xi_{i+1})=0\mod (\xi_1^2,\ldots,\xi_{i})$.  Once again, we have that $d^S(x/\xi_{i+1})=0\mod (\xi_1^2,\ldots,\xi_{i})$, and if the $t$-degree of $x$ is greater than $2^{i+2}-2$, then $x/\xi_{i+1}$ is in the image of $d^S$.\footnote{The kernel of $d^S$ modulo its image is generated by $1,\xi_{i+1}$, which live in $t$-degree less than or equal to $2^{i+2}-2$.}  In this case, $d^S(x/\xi_{i+1})=0$, so $x/\xi_{i+1}$ lifts to a class $x^{\prime}\in(\mc{A}_{\zz})_*$, and we must have that $d_{2^{i+1}-1}(\alpha^{2^{i}}x^{\prime})=x$, giving a lift of the differential which targeted $x$ in $\mo{BAHSS}_{\zz_{(2)}}^{\sss}/2$ to $\mo{BAHSS}_{\zz_{(2)}}^{\sss}$.  Therefore, for any $x\in E_{2^{i+1}-1}^{*,>2^{i+1}-2}(\mo{BAHSS}_{\zz_{(2)}}^{\sss})$ becoming a boundary in $\mo{BAHSS}_{\zz_{(2)}}^{\sss}/2$, we can lift the differential to $\mo{BAHSS}_{\zz_{(2)}}^{\sss}$, except of course if $x$ lives above $\alpha^{2^{i+1}}\xi_{i+1}^{2}$, in which case there is no class that could target $x$.  This shows that the map $$E_{2^{i+1}}^{*,t}(\mo{BAHSS}_{\zz_{(2)}}^{\sss})\to E_{2^{i+1}}^{*,t}(\mo{BAHSS}_{\zz_{(2)}}^{\sss}/2)$$ is zero for $t\leq 2^{i+2}-2$, and is injective for $t>2^{i+2}-2$, with image generated by all of the classes generating the image on the previous page which were cycles, which can be identified to be the $\fff_2[\xi_{i+2},\xi_{i+3},\ldots][\alpha^{2^{i+1}}]$-submodule generated by torsion products $\xi_{i_1}*\xi_{i_2}\ldots *\xi_{i_k}$ with $k\geq 1$, $i+1\leq i_1<i_2<\ldots <i_k$, proving the inductive step.

Unpacking the step just computed, the only point of failure for injectivity of $$E_{2^{i+1}}^{*,t}(\mo{BAHSS}_{\zz_{(2)}}^{\sss})\to E_{2^{i+1}}^{*,t}(\mo{BAHSS}_{\zz_{(2)}}^{\sss}/2)$$ above the line $t=0$ are the classes $\alpha^{k2^{i+1}}\xi_{i+1}^2$, since $\alpha^{k2^{i+1}-2^i}\xi_{i+1}$ does not lift to $E_{2^{i+1}-1}(\mo{BAHSS}_{\zz_{(2)}}^{\sss})$.  All of the leftover classes we have, both the $E_n^{*,0}$-line and all of the $\alpha^{k2^{i}}\xi_i^2$ terms leftover by these constructions map to $0$ in the mod 2 BAHSS, and thus cannot support non-trivial differentials to anything in the range where we map injectively into $E_{n}^{*,*}(\mo{BAHSS}_{\zz_{(2)}}^{\sss}/2)$.  Thus, for degree reasons, they must only support differentials between themselves.  Since all of the leftover $\alpha^{k2^i}\xi_i^2$ terms have even total degree,\footnote{That is, $t$-degree minus $s$-degree.} they cannot hit eachother under differentials, and hence cannot support any nontrivial differentials.  Since the spectral sequence converges to $\zz_{(2)}$, this leaves only one possibility: these classes are killed off as the targets of differentials starting at the 0-line.  For degree reasons, the only way this can happen is if, on $E_{2^{i}-2}$, $d_{2^{i}-2}(2^{i-2}\beta_{k})=\alpha^{k-2^{i-2}}\xi_{i-1}^2$ for $k\geq 0$ with $2^{i-2}|k$ ($i\geq 2$).  This is the statement (5), which combined with our previous observations proves (2) and (4).  Statement (3) follows immediately from (2), and the first statement follows by degree reasons, statement (2) and comparison map from $\mo{BAHSS}_{\zz_{(2)}}^{\sss}$ to $\mo{BAHSS}_{\zz_{(2)}}^{\sss}/2$.
\end{proof}
\begin{center}
	\begin{sseqpage}[title={$E_6(\mo{BAHSS}_{\zz_{(2)}}^{\sss})$ in low degrees}, axes type = center,  classes = {draw = none },
		x range = {-16}{0}, y range = {0}{14}, x axis origin = {1}, y tick gap = {-0.5 cm}, x label = { s }, y label = { t }, x label style={ yshift = -20pt }, y label style={rotate=270, yshift=360pt, xshift=23pt}, class placement transform={scale=2.9}, xscale=0.7, yscale=0.5
		]
		\class["\zz_{(2)}"](0,0)
		\class["\zz/2\zz \beta_2"](-7,0)
		\class["\zz/4\zz\beta_4"](-15,0)
		\class["\xi_2^2"](0,6)
		\class["\alpha^2\xi_2^2"](-8,6)
		\class["\alpha^4\xi_2^2"](-16,6)
		\class["\xi_3^2"](0,14)
		\class["\alpha^2\xi_3^2"](-8,14)
		\class["\alpha^4\xi_3^2"](-16,14)
		\d6(-7,0)(0,6)
		\d6(-15,0)(-8,6)
	\end{sseqpage}
\end{center}
\begin{theorem}[Theorem \ref{th1.1}]\label{th3.14}
	$$\THH_*(\tmf,\zz_{(2)})=\begin{cases}
		\zz_{(2)} & \qquad \text{if } *=0,9,13,22\\
		\zz/2^k\zz & \qquad \text{if } *=2^{k+3}r-1, 2^{k+3}r-1+9, 2^{k+3}r-1+13, 2^{k+3}r-1+22,\\
		0 & \qquad \text{otherwise},
	\end{cases}$$
for all $k>0$ and $r$ odd.
\end{theorem}
\begin{proof}
We begin with the Bockstein spectral sequence arising from $\THH(\tmf,\zz_{(2)}^{bok})$, which by Example \ref{example3.4} has $E_1$-page $\fff_2[\beta,\tilde{v_0}]\otimes_{\fff_2}\Lambda(\lambda_1,\lambda_2,\lambda_3)$, with $|\beta|=(-16,0)$, $\lambda_1=(-9,0)$, $|\lambda_2|=(-13,0)$, $|\lambda_3|=(-15,0)$ and $|\tilde{v_0}|=(1,1)$.  None of the $\lambda_i$ can support any nontrivial differentials for degree reasons, so either the spectral sequence has degenerated, or else $d_n(\beta)\neq 0$ for some $n\geq 1$.  If the spectral sequence degenerated already, then $\THH_i(\tmf,\zz_{(2)})$ would have a torsion-free summand for infinitely many $i$.  However, by Lemma \ref{lem3.10}, this cannot be the case, so we must have that $d_n(\beta)=\lambda_3\tilde{v_0}^n$ for some $n$.  By similar reasoning, each power $\beta^{2^k}$ must support some differential, so we get some strictly increasing sequence $\{n_k\}_{k\geq 0}$ such that $d_{n_{k}}(\beta^{2^{k}})=\beta^{2^{k}-1}\lambda_3\tilde{v_0}^{n_k}$, and together with multiplicativity, these generate all of the differentials in this spectral sequence.  Our job now is to determine exactly what this sequence $n_r$ is.  This discussion shows that $\THH_i(\tmf,\zz_{(2)})$ will be $\zz_{(2)}$ for $i=0,9,13,22$, and will be of the form $\zz/2^{n_k}\zz$ if $i=2^{k+4}r-1+j$ for $r$ odd, $k\geq 0$ and $j\in\{0,9,13,22\}$.  This corresponds to the $1,\lambda_1,\lambda_2,\lambda_1\lambda_2$-multiplies of $\beta^{2^kr-1}\lambda_3$.  In particular, aside from the four exceptional cases, $\THH_i(\tmf,\zz_{(2)})$ is only nonzero for $i\equiv 4,8,12,15 \mod 16$.
\begin{center}
	\begin{sseqpage}[title={$E_1$-page in low degrees}, axes type = center,  classes = {draw = none },
		x range = {-16}{2}, y range = {0}{2}, x axis origin = {1}, y tick gap = {-0.2 cm}, x label = { s }, y label = { t }, x label style={ yshift = -15pt }, y label style={rotate=270, yshift=320pt, xshift=23pt}, class placement transform={scale=2.9}, xscale=0.7, yscale=1
		]
		\class["1"](0,0)
		\class["\tilde{v}_0"](1,1)
		\class["\tilde{v}_0^2"](2,2)
		\class["\lambda_1"](-9,0)
		\class["\tilde{v}_0\lambda_1"](-8,1)
		\class["\tilde{v}_0^2\lambda_1"](-7,2)
		\class["\lambda_2"](-13,0)
		\class["\tilde{v}_0\lambda_2"](-12,1)
		\class["\tilde{v}_0^2\lambda_2"](-11,2)
		\class["\lambda_3"](-15,0)
		\class["\tilde{v}_0\lambda_3"](-14,1)
		\class["\tilde{v}_0^2\lambda_3"](-13,2)
		\class["\beta"](-16,0)
		\class["\tilde{v}_0\beta"](-15,1)
		\class["\tilde{v}_0^2\beta"](-14,2)
	\end{sseqpage}
\end{center}
\indent We now turn our attention to the Brun-Atiyah-Hirzebruch spectral sequence and use our computation Proposition \ref{prop3.9} more substantially.  We will not compute the BAHSS in its entirety, instead we will focus our efforts on computing $\THH_{2^{i}-1}(\tmf,\zz_{(2)})$ for $i\geq 4$.  The $E_1$-page is given by $$E_1^{s,t}(\mo{BAHSS}_{\zz_{(2)}}^{\tmf})=\THH_{-s}(\zz_{(2)},\pi_{t}(\zz_{(2)}\otimes_{\tmf}\zz_{(2)})).$$  We will use notation as before for classes in $\THH_*(\zz_{(2)})$ and $\THH_*(\zz_{(2)},\fff_2)$, with the same abuse of notation as in Example \ref{ex3.12}.

The classes on the $E_1$-page which contribute to the total degree $2^{i}-1$ term are (for $i\geq 4$)\footnote{If a class has a negative power of $\alpha$, we set it to be 0 as a convention.} $\beta_{2^{i-2}}$, $\alpha^{2^{i-2}-2}\gamma x^2$, $ \alpha^{2^{i-2}-3}\gamma xy$, $ \alpha^{2^{i-2}-3}xw$, $ \alpha^{2^{i-2}-4}\gamma x^3y$, $ \alpha^{2^{i-2}-4} x^3w$,$ \alpha^{2^{i-2}-4} xv$, and finally $\alpha^{2^{i-2}-5}x^3v$, with all of these classes generating a copy of $\zz/2\zz$, except for $\beta_{2^{i-2}}$ which generates a copy of $\zz/2^{i-1}$.  The BAHSS has no nontrivial $d_1$-differentials for degree reasons.  From the comparison map $E_i(\mo{BAHSS}_{\zz_{(2)}}^{\sss})\to E_i(\mo{BAHSS}_{\zz_{(2)}}^{\tmf})$, we find that $d_2(\beta_k)=\alpha^{k-1}x$, which together with multiplicativity, determines most of the $d_2$-differentials.   By multiplicativity, all of the classes from the $E_1$-page with total degree $2^{i}-1$ either support a nontrivial differential, or are hit by a nontrivial differential, with the sole exception being the class $\alpha^{2^{i-2}-4}\gamma x^3y$ on account of the fact that we have $d_2(\alpha^{2^{i-2}-4}\gamma x^3y)=\alpha^{2^{i-2}-4}x^4y=0$, from the relation $x^4=0$.

Next, consider the induced map to the mod 2 BAHSS considered in Example \ref{ex3.11}.  The class $\alpha^{2^{i-2}-3}\gamma u$ (which survives to $E_3$ since $ux=0$) maps to the class $\alpha^{2^{i-2}-3}\gamma \xi_2\xi_1^6$ on the $E_3$-page.  By Example \ref{ex3.11}, $d_3(\alpha^{2^{i-2}-3}\gamma \xi_2\xi_1^6)=\alpha^{2^{i-3}-4}\gamma\xi_2^2\xi_1^6$, so we must have a class in the BAHSS we are analyzing lifting $\alpha^{2^{i-3}-4}\gamma\xi_2^2\xi_1^6$ which is hit by $\alpha^{2^{i-2}-3}\gamma u$ on the $E_3$-page, and the only possible such class is $\alpha^{2^{i-2}-4}\gamma x^3y$, which therefore must vanish on $E_4$.

Finally, since every class with $t$-degree between $1$ and $5$ vanished on $E_3$, there is no class leftover to hit $\alpha^{2k}y$ for $k\geq 0$ before the $E_6$-page, so comparison with the case of the sphere tells us that $d_6(2\beta_{2k+2})=\alpha^{2k}y$ for all $k\geq 0$, leaving us with a $\zz/2^{k}\zz$ as the only group in total degree $2^{k+3}-1$ on the $E_7$-page, telling us that $n_k$ is at most $k$.  Since our analysis of the Bockstein spectral sequence tells us that $n_k$ is strictly increasing, and $n_1\geq 1$, this forces $n_k=k$, proving the claim.
\end{proof}

\newpage
\section{THH of Quotients of $\ell$}

\indent\indent In this section, we compute the topological Hochschild homology of quotients $\ell/v_1^n$ of $\ell$, in terms of the topological Hochschild homology of $\ell$, which was computed by Angeltveit-Hill-Lawson in \cite{angeltveit2009topological}.  We will show that $\pi_*\THH(\ell/v_1^n)\simeq \pi_*\THH(\ell,\ell/v_1^n)\otimes_{\zz_{(p)}}\Gamma[x]$, the tensor product of the topological Hochschild homology of $\ell$ with coefficients in $\ell/v_1^n$ with a divided power algebra on a generator $x$ in degree $2n(p-1)+2$.  We start by computing $\THH(\ell/v_1^n,\zz_{(p)})$.
\begin{lemma}\label{lem4.1}
	For $n>1$, we have that $$\zz_{(p)}\otimes_{\ell/v_1^n}\zz_{(p)}\simeq (\zz_{(p)}\otimes_{\ell}\zz_{(p)})\otimes_{\zz_{(p)}}(\zz_{(p)}\otimes_{\zz_{(p)}\otimes_{\ell}\ell/v_1^n}\zz_{(p)})$$ as $\bb{E}_\infty$-algebras.
\end{lemma}
\begin{proof}
We have a diagram of $\bb{E}_{\infty}$-rings where all squares are pushouts:
\begin{center}
	\begin{tikzcd}
	\ell\rar\dar & \ell/v_1^n\rar\dar & \zz_{(p)}\dar\\
	\zz_{(p)}\rar & \zz_{(p)}\otimes_{\ell}\ell/v_1^n\rar\dar & \zz_{(p)}\otimes_{\ell}\zz_{(p)}\dar\\
	& \zz_{(p)}\rar & \zz_{(p)}\otimes_{\ell/v_1^n}\zz_{(p)}.
	\end{tikzcd}
\end{center}
The map $\zz_{(p)}\otimes_{\ell}\ell/v_1^n\to \zz_{(p)}\otimes_{\ell}\zz_{(p)}$ factors over $\tau_{\leq 2(p-1)+1}\left(\zz_{(p)}\otimes_{\ell}\ell/v_1^n\right)\simeq \zz_{(p)}$, so that the cospan $\zz_{(p)}\leftarrow \zz_{(p)}\otimes_{\ell}\ell/v_1^n\rightarrow \zz_{(p)}\otimes_{\ell}\zz_{(p)}$ may be rewritten as the tensor product over $\zz_{(p)}$ of the cospans $\zz_{(p)}\leftarrow \zz_{(p)}\otimes_{\ell}\ell/v_1^n\rightarrow \zz_{(p)}$ and $\zz_{(p)}\leftarrow \zz_{(p)}\rightarrow \zz_{(p)}\otimes_{\ell}\zz_{(p)}$.  Since colimits commute, we find that the pushout of our original square, $\zz_{(p)}\otimes_{\ell/v_1^n}\zz_{(p)}$, is isomorphic to the pushout of our first cospan tensored over $\zz_{(p)}$ with the pushout of our second cospan, which is precisely $(\zz_{(p)}\otimes_{\zz_{(p)}\otimes_{\ell}\ell/v_1^n}\zz_{(p)})\otimes_{\zz_{(p)}}(\zz_{(p)}\otimes_{\ell}\zz_{(p)})$, as claimed.
\end{proof}
Next, we examine $\zz_{(p)}\otimes_{\zz_{(p)}\otimes_{\ell}\ell/v_1^n}\zz_{(p)}$.  By an easy computation, its homotopy groups are 
\begin{center}
	\begin{equation*}
	\pi_*(\zz_{(p)}\otimes_{\ell}\ell/v_1^n)=\begin{cases}
		\zz_{(p)} & \text{if } * = 0, 2n(p-1)+1,\\
		0 & \text{otherwise.}
	\end{cases}
	\end{equation*}
\end{center}
By \cite[~Proposition 2.1]{Dundas_2018}, this $\bb{E}_{\infty}-\zz_{(p)}$-algebra is a trivial square-zero extension of $\zz_{(p)}$.  A standard then argument shows that there is an identification of graded commutative rings $\pi_*(\zz_{(p)}\otimes_{\zz_{(p)}\otimes_{\ell}\ell/v_1^n}\zz_{(p)})=\Gamma[x]$ with $x$ in degree $2n(p-1)+2$, as desired.

\begin{lemma}\label{lem4.2}
	$\THH_*(\ell/v_1^n,\zz_{(p)})\simeq \THH_*(\ell,\zz_{(p)})\otimes_{\zz_{(p)}}\pi_*(\zz_{(p)}\otimes_{\zz_{(p)}\otimes_{\ell}\ell/v_1^n}\zz_{(p)})$.
\end{lemma}
\begin{proof}
By cyclic invariance, we have
$$\THH(\ell/v_1^n,\zz_{(p)})\simeq \THH(\zz_{(p)},\zz_{(p)}\otimes_{\ell/v_1^n}\zz_{(p)}).$$
Expanding this out and applying Lemma \ref{lem4.1}, we get
\begin{align*}
\THH(\zz_{(p)},\zz_{(p)}\otimes_{\ell/v_1^n}\zz_{(p)})&\simeq \zz_{(p)}\otimes_{\zz_{(p)}\otimes_{\sss}\zz_{(p)}}(\zz_{(p)}\otimes_{\ell/v_1^n}\zz_{(p)})\\
&\simeq \zz_{(p)}\otimes_{\zz_{(p)}\otimes_{\sss}\zz_{(p)}}((\zz_{(p)}\otimes_{\ell}\zz_{(p)})\otimes_{\zz_{(p)}}(\zz_{(p)}\otimes_{\zz_{(p)}\otimes_{\ell}\ell/v_1^n}\zz_{(p)})).
\end{align*}
Finally, by rearranging the order of the tensor products, one concludes
\begin{align*}
\THH(\ell/v_1^n,\zz_{(p)})&\simeq \zz_{(p)}\otimes_{\zz_{(p)}\otimes_{\sss}\zz_{(p)}}((\zz_{(p)}\otimes_{\ell}\zz_{(p)})\otimes_{\zz_{(p)}}(\zz_{(p)}\otimes_{\zz_{(p)}\otimes_{\ell}\ell/v_1^n}\zz_{(p)}))\\
&\simeq (\zz_{(p)}\otimes_{\zz_{(p)}\otimes_{\sss}\zz_{(p)}}(\zz_{(p)}\otimes_{\ell}\zz_{(p)}))\otimes_{\zz_{(p)}}(\zz_{(p)}\otimes_{\zz_{(p)}\otimes_{\ell}\ell/v_1^n}\zz_{(p)})\\
&\simeq \THH(\zz_{(p)},\zz_{(p)}\otimes_{\ell}\zz_{(p)})\otimes_{\zz_{(p)}}(\zz_{(p)}\otimes_{\zz_{(p)}\otimes_{\ell}\ell/v_1^n}\zz_{(p)})\\
&\simeq \THH(\ell,\zz_{(p)})\otimes_{\zz_{(p)}}(\zz_{(p)}\otimes_{\zz_{(p)}\otimes_{\ell}\ell/v_1^n}\zz_{(p)}).
\end{align*} 
\end{proof}
Now, we are finally in a position to analyze the Whitehead spectral sequence for $\THH(\ell/v_1^n)$.  To start, note that $$E_1^{s,t}=\zz_{(p)}[v_1]/v_1^n\otimes_{\zz_{(p)}}\THH_{*}(\ell,\zz_{(p)})\otimes_{\zz_{(p)}}\Gamma[x],$$ with $v_1$ in bidegree $(0,2(p-1))$, $x$ in bidegree $(-(2n(p-1)+2),0)$, and $\THH_{s}(\ell,\zz_{(p)})$ living in degree $(-s,0)$.  There is a comparison map $\rho$ from the Brun-Atiyah-Hirzebruch spectral sequence for $\THH(\ell,\ell/v_1^n)$ to this one, which determines many of the differentials.  In fact,
\begin{theorem}\label{th4.3}
All of the differentials vanish on the classes coming from $\Gamma[x]$ in the $\THH(\ell/v_1^n)$ spectral sequence.  In particular, the comparison map $\rho$, together with the Leibniz rule, determine all of the differentials in the Whitehead spectral sequence, yielding $$\THH_*(\ell/v_1^n)\simeq \THH_*(\ell,\ell/v_1^n)\otimes_{\zz_{(p)}}\Gamma[x],$$ for some class $x$ in degree $2n(p-1)+2$.
\end{theorem}
\begin{proof}
	In the case $n\neq 1\mod p^2$, the theorem can be proven by looking only at the spectral sequences we have already constructed.
	
To prove this theorem in general, we use the May-type spectral sequence, combined with the following lemma, whose proof is adapted from \cite[~Lemma 4.1]{lee2023topological}:
\begin{lemma}\label{lem4.4}
	Suppose $k$ is a discrete ring, and $R$ is a connective (possibly graded) $\bb{E}_{2}$-$k$-algebra with $\pi_*(R)=k[x]/x^n$, on some class $x$ in positive even degree, and $R$ admits an $\bb{E}_2$-algebra map from a ring $S$ with $\pi_*(S)=k[x]$.  Then, we have an equivalence of (graded) $\bb{E}_1$-$k$-algebras $\THH(R)=\THH(k)\otimes_k \mo{HH}(R/k)$.
\end{lemma}
\begin{proof}[Proof of lemma]
We have $k[x]=k\otimes_{\sss}\sss[x]$.  Now, as an $\bb{E}_1$-algebra $$R=k[x]\otimes_{k[x^n]}k\simeq (k\otimes_{\sss}\sss[x])\otimes_{k\otimes_{\sss}\sss[x^n]}(k\otimes_{\sss}\sss)\simeq k\otimes_{\sss}\sss[x]/x^n,$$ where $\sss[x]/x^n$ denotes $\sss[x]\otimes_{\sss[x^n]}\sss$.  Since $\THH$ commutes with tensor products, there are equivalences of (graded) spectra, $$\THH(R)\simeq \THH(k)\otimes_{\sss}\THH(\sss[x]/x^n)\simeq \THH(k)\otimes_k k\otimes_{\sss}\THH(\sss[x]/x^n)\simeq \THH(k)\otimes_k \mo{HH}((k[x]/x^n)/k).$$  Since $k[x]/x^n=\tau_{\leq n|x|-1}k[x]$, $R$ inherits a canonical $\bb{E}_2$-$k$-algebra structure as this truncation.  We can give $x$ a new (positive) grading 1, to make $\sss[x]$ into a nonnegatively graded $\bb{E}_2$-ring spectrum, that is to say an $\bb{E}_2$-algebra object in $\mo{Sp}^{\zz^{ds}_{\geq 0}}$.  There is a thick $\otimes$-ideal $\mc{I}$ of $\mo{Sp}^{\zz^{ds}_{\geq 0}}$ generated by elements concentrated in grading $\geq n$.  Quotienting out this $\otimes$-ideal gives a symmetric monoidal functor $\mo{Sp}^{\zz^{ds}_{\geq 0}}\to \mo{Sp}^{\zz_{\geq 0}^{ds}}$$/\mc{I}$, whose right adjoint is then lax symmetric monoidal by \cite[~Corollary 7.3.2.7]{HA}.  Composing these two functors gives a functor which sends our graded $\sss[x]$ to a graded $\bb{E}_2$-algebra with underlying $\bb{E}_2$-ring $\sss[x]/x^n$, as desired.  Using this grading, we can get, from the $\bb{E}_2$-map in $\mo{Sp}^{\zz^{ds}_{\geq 0}}$, a map $\sss[x]\to k[x]$.  Applying the endofunctor we just described, we get an $\bb{E}_2$-algebra map $\sss[x]/x^n\to k[x]/x^n$, which upgrades our isomorphism above to an $\bb{E}_2$-algebra isomorphism, ensuring that the induced map on $\THH$ is an isomorphism of $\bb{E}_1$-algebras.
\end{proof}
We wish to apply this in our case.  Choosing $m\in\zz_{\geq 0}$ sufficiently large ($m>2n(p-1)+2$), we get a $t$-structure $\tau^m$ on graded spectra such that $\tau^m_{\geq 0}(\pi_*(\ell/v_1^n))=\zz_{(p)}$ concentrated in degree 0, which shows that $\pi_*(\ell/v_1^n)$ is a graded $\bb{E}_{\infty}$-$\zz_{(p)}$-algebra.  Now, we can apply the above theorem to write the $E_1$-page of the May-type spectral sequence as $$E_1^{s,t}=\THH_*(\zz_{(p)})\otimes_{\zz_{(p)}}\mo{HH}_*((\zz_{(p)}[\tilde{v_1}]/\tilde{v_1}^n)/\zz_{(p)}).$$  A standard calculation\footnote{For instance, one can use that $\mo{HH}((\zz_{(p)}[x]/x^n)/\zz_{(p)})\simeq \mo{HH}(\zz_{(p)}[x]/\zz_{(p)})\otimes_{\mo{HH}(\zz_{(p)}[x^n]/\zz_{(p)})}\zz_{(p)}$.} shows that $$\mo{HH}_*((\zz_{(p)}[\tilde{v_1}]/\tilde{v_1}^n)/\zz_{(p)})\simeq \Lambda_{\zz_{(p)}}[\sigma\tilde{v_1}]\otimes_{\zz_{(p)}} \Gamma_{\zz_{(p)}}[\sigma^2\tilde{v_1}^n]\otimes_{\zz_{(p)}} \zz_{(p)}[\tilde{v_1}]/(\tilde{v_1}^n),$$ where $\sigma^2\tilde{v_1}^n$ is a class in bidegree $(-2,2n(p-1))$, and $\sigma v_1$ is a class in bidegree $(-1,2(p-1))$.

The terms coming from $\THH_n(\zz_{(p)})$ live in bidegree $(-n,0)$, and $\tilde{v_1}$ lives in bidegree $(0,2(p-1))$.  One notes that $\sigma^2(\tilde{v_1}^n)$ is the only class in total degree $(-2,2n(p-1))$ and nothing lives in degree $(0,2n(p-1)+1)$.  Thus, $\sigma^2(\tilde{v_1}^n)$ must vanish under the differentials on every page, and this is the same class corresponding to $x$ in the Brun-Atiyah-Hirzebruch spectral sequence.

Next, note that for all classes $a$, with bigrading $|a|=(s,t)$, we have that \\$t\leq -n(p-1)s+2(n-1)(p-1)$ for $s$ even, and $t\leq -n(p-1)(s+1)+2n(p-1)$ for $s$ odd.  Furthermore, $t$ is maximized with respect to $s$ for $s\leq 0$ even by $\tilde{v_1}^{n-1}x^{(-\f{s}{2})}$, and for $s$ odd by $\sigma\tilde{v_1}\tilde{v_1}^{n-1}x^{(-\f{s+1}{2})}$.  In particular, any differential off of $x^{(k)}$ on the $E_r$-page would have to hit a class in bidegree $(-2k+r+1, 2kn(p-1)+r)$.  However, for $r$ odd, $-2k+r+1\geq -2k+2$, and thus every class with $t> 2n(p-1)(k-1)+2(n-1)(p-1)$ vanishes in this $s$ degree.  Therefore, the target of $d_r$ on $x^{(k)}$ is 0.  If $r$ is even, then $-2k+r+1\geq -2k+3$, and thus if $t>2n(p-1)(k-2)+2n(p-1)=2n(p-1)(k-1)$, the classes in degree $(-2k+r+1,t)$ vanish, and again, the target of $d_r((\sigma^2 (\tilde{v_1}^n))^{(k)})$ vanishes.  We have shown that the $(\sigma^2 (\tilde{v_1})^n)^{(k)}$, are all permanent cycles, and then so are the $x^{(k)}$ from the Brun-Atiyah-Hirzebruch spectral sequence.

Note that there are no other nonzero terms in the May-type spectral sequence with total degree $k(2n(p-1)+2)$, and higher filtration degree than $(\sigma^2 (\tilde{v_1}^n))^{(k)}$.  Thus, there can be no nontrivial multiplicative extensions supported on these classes, and $x\mapsto \sigma^2 (v_1^n)$ determines a map of graded commutative $\zz_{(p)}$-algebras $\Gamma_{\zz_{(p)}}[x]\to \THH_*(\ell/v_1^n)$, with $x$ a class in degree $2n(p-1)+2$.  This gives us a map of graded commutative $\zz_{(p)}$-algebras 
$$\THH_*(\ell,\ell/v_1^n)\otimes_{\zz_{(p)}}\Gamma_{\zz_{(p)}}[x]\to \THH_*(\ell/v_1^n).$$
Since the $E_1$-page of the Whitehead spectral sequence for $\THH(\ell/v_1^n)$ is multiplicatively generated by the image of $\THH_*(\ell,\pi_*(\ell/v_1^n))$ under $\rho$ and the classes $x^{(k)}$ (which have just been shown to be permanent cycles), all of the nontrivial differentials appearing in this Whitehead spectral sequence arise from the map $\THH_*(\ell,\tau_{\geq *}\ell/v_1^n)\to\THH_*(\ell/v_1^n,\tau_{\geq *}\ell/v_1^n)$ together with the Leibniz rule.  In particular, it follows that the algebra map $\THH_*(\ell,\ell/v_1^n)\otimes_{\zz_{(p)}}\Gamma_{\zz_{(p)}}[x]\to \THH_*(\ell/v_1^n)$ must be an isomorphism, as claimed.
\end{proof}
\begin{remark}
The notations $\sigma x$ and $\sigma^2 x$ for certain classes $x$ were used above, which deserves an explanation.  Let $\mc{C}$ be a presentably symmetric monoidal stable category, and let $R$ be an $\bb{E}_1$-algebra object in $\mc{C}$.  In \cite[~\textsection A]{hahn2022redshift}, a map is constructed $$\Sigma\cofib(1_{\mc{C}}\to R)\to\THH_{\mc{C}}(R),$$ where we denote for the moment $\THH_{\mc{C}}(R)$ the topological Hochschild homology internal to the category $\mc{C}$.  For a class $w\in \pi_*(R)$, we denote by $\sigma w$ the image of this class under the composite $\Sigma R\to \Sigma\cofib(1_{\mc{C}}\to R)\to \THH_{\mc{C}}(R)$, $\sigma w\in \pi_{*+1}(\THH_{\mc{C}}(R))$.  If $x\in \pi_{*}(1_{\mc{C}})$ has a lift to a class $\tilde{x}\in\pi_{*+1}(\cofib(1_{\mc{C}}\to R))$, then we write $\sigma^2x$ for the image of $\tilde{x}$ under $\Sigma\cofib(1_{\mc{C}}\to R)\to\THH_{\mc{C}}(R)$, which lives in $\pi_{*+2}(\THH_{\mc{C}}(R))$.  Although the class denoted $\sigma^2\tilde{v_1}^{n}$ above does not actually arise in this way, it is lifted from a class in $\THH_{\zz_{(p)}[\tilde{v_1}]-Mod}(\zz_{(p)}[\tilde{v_1}]/(\tilde{v_1}^n))$, which is constructed in this fashion, so we abusively denote the class $\sigma^2\tilde{v_1}^n$.  These naming conventions will make a reprise in the sequel.
\end{remark}
\newpage 
\section{The General Case}
We remark that many of the constructions in the last section admit a generalization.
\begin{theorem}\label{th5.1}
 	Suppose that $R$ is a connective $\bb{E}_{m}$-ring spectrum for some $m\geq 4$, and $x\in \pi_{*}(R)$ is a positive degree class such that, for some fixed $k>1$, there is an $\bb{E}_3$-$R$-algebra $S$, such that the unique algebra map $R\to S$ fits into a fiber sequence of $R$-modules $$\Sigma^{k|x|}R\xrightarrow{x^k}R\to S.$$   Then, there is an equivalence of $\pi_0(R)$-modules $$\THH(S,\pi_0(S))\simeq \THH(R,\pi_0(R))\otimes_{\pi_0(R)}(\pi_0(R)\otimes_{\pi_0(R) \otimes_{R} S}\pi_0(R)).$$
\end{theorem}
\begin{proof}
By proposition \ref{proposition2.7}, using that $\pi_0(R)=\pi_0(S)$, $$\THH(S,\pi_0(R))\simeq \THH(\pi_0(R),\pi_0(R)\otimes_{S}\pi_0(R)).$$ There is a string of equivalences
\begin{align*}
\pi_0(R)\otimes_{S}\pi_0(R)&\simeq (\pi_0(R)\otimes_{\sss}\pi_0(R)) \otimes_{S\otimes_{\sss}S} S\\
&\simeq (\pi_0(R)\otimes_{\sss}\pi_0(R)) \otimes_{S\otimes_{\sss}S} (S\otimes_{R} S)\otimes_{(S\otimes_{R} S)}S\\
&\simeq ((\pi_0(R)\otimes_{\sss}\pi_0(R)) \otimes_{S\otimes_{\sss}S} ((S\otimes_{\sss}S)\otimes_{R\otimes_{\sss}R}R))\otimes_{(S\otimes_{R} S)}S\\
&\simeq  ((\pi_0(R)\otimes_{\sss}\pi_0(R)) \otimes_{R\otimes_{\sss}R} R)\otimes_{(S\otimes_{R} S)}S\\
&\simeq (\pi_0(R)\otimes_R \pi_0(R))\otimes_{(S\otimes_{R} S)}S\\
&\simeq (\pi_0(R)\otimes_R \pi_0(R))\otimes_{\pi_0(R)\otimes_{R} S}(\pi_0(R)\otimes_{R} S)\otimes_{(S\otimes_{R} S)}S\\
&\simeq (\pi_0(R)\otimes_R \pi_0(R))\otimes_{\pi_0(R)\otimes_{R} S} \pi_0(R).
\end{align*}
The first, third, fourth, fifth, and seventh equivalences hold by \cite[~Theorem 5.1.4.10]{HA}. The second equivalence follows by simple rewriting $S$ as $(S\otimes_{R}S)\otimes_{S\otimes_{R}S}S$, and the fact that $S\otimes_{\sss}S\to S$ factors over $S\otimes_{R}S\to S$.  Finally, the sixth equivalence uses that $S\otimes_{R}S\to \pi_0(R)\otimes_{R}\pi_0(R)$ factors as $S\otimes_{R}S\to \pi_0(R)\otimes_{R}S\to \pi_0(R)\otimes_{R}\pi_0(R)$. 

Now, by the assumption $S\simeq R/x^k$ as $R$-modules, it follows that the map $S\to \pi_0(S)=\pi_0(R)$ factors over some $R$-module map $S\to R/x^{k-1}\to \pi_0(R)$.  Thus, $\pi_0(R)\otimes_{R} S\to \pi_0(R)\otimes_R \pi_0(R)$ factors as $\pi_0(R)\otimes_{R} S\to\pi_0(R)\otimes_{R} R/x^{k-1}\to \pi_0(R)\otimes_R \pi_0(R)$.  Since $\pi_0(R)\otimes_{R} R/x^{k-1}$ has homotopy groups given by 
$$\pi_*(\pi_0(R)\otimes_{R} R/x^{k-1})=\begin{cases}
	\pi_0(R) &\qquad \text{if } *=0,\, (k-1)|x|+1,\\
	0 &\qquad \text{otherwise},
	\end{cases}$$
we see that $\pi_0(R)\otimes_{R} S\to \pi_0(R)\otimes_{R} R/x^{k-1}$ factors over $\tau_{\leq (k-1)|x|+1}(\pi_0(R)\otimes_{R} S)\simeq \pi_0(R)$.  This implies that $\pi_0(R)\otimes_{R}\pi_0(R)$, as a right $\pi_0(R)\otimes_{R} S$-module, is equivalent to \\$(\pi_0(R)\otimes_{R} \pi_0(R))\otimes_{\pi_0(R)}\pi_0(R)$, with the induced right module structure on $\pi_0(R)$.  Thus, 
\begin{align*}
\pi_0(R)\otimes_{S}\pi_0(R)&\simeq (\pi_0(R)\otimes_R \pi_0(R))\otimes_{\pi_0(R)\otimes_{R} S} \pi_0(R)\\
&\simeq ((\pi_0(R)\otimes_{R} \pi_0(R))\otimes_{\pi_0(R)} \pi_0(R)) \otimes_{\pi_0(R) \otimes_{R} S}\pi_0(R)\\
&\simeq (\pi_0(R)\otimes_{R} \pi_0(R))\otimes_{\pi_0(R)} (\pi_0(R) \otimes_{\pi_0(R) \otimes_{R} S}\pi_0(R)).
\end{align*}
The first two equivalences follow by our discussion above, and the final equivalence uses that $\pi_0(R)\otimes_{R}S\to \pi_0(R)\otimes_{R}\pi_0(R)$ factors over $\pi_0(R)$.  This derived tensor product in $\pi_0(R)$-modules can be computed as a tensor product on underlying modules, since the second module is flat (in fact free).  Indeed, from the cofiber sequence $\Sigma^{k|x|}R\xrightarrow{x^k}R\to S$, we can tensor this with $\pi_0(R)$ to find that $\pi_0(R)\otimes_{R} S\simeq \pi_0(R)\oplus \Sigma^{k|x|+1}\pi_0(R)$ as a $\pi_0(R)$-module.  We then have a periodic resolution of $\pi_0(R)$ from this class in degree $k|x|+1$, which allows us to see that $$\pi_0(R) \otimes_{\pi_0(R) \otimes_{R} S}\pi_0(R)\simeq \bigoplus_{r\geq 0}\Sigma^{r(k|x|+2)}\pi_0(R)$$ as a $\pi_0(R)$-module.

Now, we apply the Whitehead filtration to $\pi_0(S)\otimes_{S}\pi_0(S)$, and examine the Brun-Atiyah-Hirzebruch spectral sequence for $\THH(S, \pi_0(S))$.  Note that this spectral sequence is multiplicative, since by assumption, $S$ is an $\bb{E}_3$-$R$-algebra, so that the maps $S\to \pi_0(S)=\tau_{\leq 0}S$ are $\bb{E}_3$-algebra maps.  This implies that $(\pi_0(S)\otimes_{S}\pi_0(S))$ is an $\bb{E}_2$-$\pi_0(R)$-algebra by \cite[~Proposition 7.1.2.6]{HA}. We have a map $$\THH(\pi_0(R), \tau_{\geq *}(\pi_0(R)\otimes_{R}\pi_0(R)))\to \THH(\pi_0(S), \tau_{\geq *}(\pi_0(S)\otimes_{S}\pi_0(S)))$$ which descends to a map on the associated Brun-Atiyah-Hirzebruch spectral sequences.  Since $\pi_0(S)\simeq \pi_0(R)$ are $\bb{E}_{\infty}$-rings, lemma \ref{lemma2.6} applies and both of these Atiyah-Hirzebruch spectral sequences are multiplicative.  The $E_1$-page of the target is multiplicatively generated by the classes in the image of this map, together with classes generating copies of $\pi_0(R)$ in degrees $(0,r(k|x|+2))$ for $r>0$.  Since there are no nonzero classes in bidegree $(s,t)$ for $s>0$, the differentials vanish on these classes, and there are no multiplicative extension problems between them.  The map from $\THH(\pi_0(R), \tau_{\geq *}(\pi_0(R)\otimes_{R}\pi_0(R)))$ determines the extension problems on the classes in its image, and this determines all of the additive extension problems, since any nonzero class $a$ in the image of this map multiplies with any nonzero class $b$ coming from $(\pi_0(R) \otimes_{\pi_0(R) \otimes_{R} S}\pi_0(R))$ to a nonzero class.\footnote{In fact, we have maps of underlying graded $\pi_0(S)$-algebras $\THH_*(R,\pi_0(R))\to \THH_*(S,\pi_0(S))$, and $\pi_*(\pi_0(R)\otimes_{\pi_0(R)\otimes_{R}S}\pi_0(R))\to \THH_{*}(S,\pi_0(S))$, which extends to an algebra homomorphism on their graded tensor product, which Theorem \ref{th5.1} shows to be an isomorphism, showing there are no multiplicative extension problems either.}

This establishes the claim on the level of homotopy groups.  For the full claim, note that we have a map $$\THH(\pi_0(R),\pi_0(R)\otimes_{S}\pi_0(R))\to \THH(\pi_0(R),\pi_0(R) \otimes_{\pi_0(R) \otimes_{R} S}\pi_0(R))\to \pi_0(R) \otimes_{\pi_0(R) \otimes_{R} S}\pi_0(R),$$ which admits a splitting $\varphi:\bigoplus_{r\geq 0}\Sigma^{r(k|x|+2)}\pi_0(R)\to \THH(\pi_0(R),\pi_0(R)\otimes_{S}\pi_0(R))$.  Since $\THH(\pi_0(R),\pi_0(R)\otimes_{S}\pi_0(R))$ admits a $\THH(\pi_0(R),\pi_0(R)\otimes_{R}\pi_0(R))$-module structure coming from the natural map, $\varphi$ extends to a map $$\THH(\pi_0(R),\pi_0(R)\otimes_{R}\pi_0(R))\otimes_{\pi_0(R)}\pi_0(R) \otimes_{\pi_0(R) \otimes_{R} S}\pi_0(R) \to \THH(\pi_0(R),\pi_0(R)\otimes_{S}\pi_0(R)),$$ which provides our desired equivalence.
\end{proof}
\begin{remark}\label{rem5.2}  As the proof indicates, we can replace $x^k$ by any class $x$ in positive degree such that $R/x$ admits some $\bb{E}_3$-$R$-algebra, and such that $(R/x\otimes_{R}\pi_0(R))\to (\pi_0(R) \otimes_R \pi_0(R))$ factors over $(R/x\otimes_{R}\pi_0(R))\to \tau_{\leq 0}((R/x\otimes_{R}\pi_0(R)))\simeq\pi_0(R)$.
	\end{remark}
\begin{corollary}\label{cor5.3}
	Let $R$, $x$ and $S$ be as in Theorem \ref{th5.1}.  Then, the $E_1$-page of the Whitehead spectral sequence computing $\THH(S)$ is isomorphic to the $E_1$-page of the Atiyah-Hirzebruch spectral sequence computing $\THH(R,S)$ tensored over $\pi_0(R)$ with $\bigoplus_{r\geq 0}\pi_0(R)\cdot a_r$, where $a_r$ is a class in bidegree $(-r(k|x|+2),0)$.
\end{corollary}
\begin{proof}
This follows from the string of equivalences
\begin{align*}
\THH(S,\pi_*(R))&\simeq \THH(S,\pi_0(R))\otimes_{\pi_0(R)} \pi_*(R)\\
&\simeq \THH(R,\pi_0(R))\otimes_{\pi_0(R)}(\pi_0(R) \otimes_{\pi_0(R) \otimes_{R} S}\pi_0(R))\otimes_{\pi_0(R)}\pi_*(R)\\
&\simeq \THH(R,\pi_*(R))\otimes_{\pi_0(R)}(\pi_0(R) \otimes_{\pi_0(R) \otimes_{R} S}\pi_0(R)),
\end{align*}
where Theorem \ref{th5.1} was used to get the second equivalence.
\end{proof}
\begin{remark}\label{rem5.4}  We don't know whether or not $$\THH_*(S)\simeq \THH_*(R,S)\otimes_{\pi_0(S)} \pi_*(\pi_0(S)\otimes_{\pi_0(S)\otimes_{R}S}S)$$ in general, although this does seem to hold in many cases.
\end{remark}
Together with the map from $\THH(R,\tau_{\geq *}(S))$, this means that to understand $\THH(S)$, we need only understand $\THH(R,S)$, and how the differentials act on the classes $a_r$.

\newpage
\section{The THH of $\ksc_2$}
Using the results of section \textsection 5, we will compute $\THH(\ksc_2)$ and show that it fits into the same overall framework.  Recall (cf. \cite{BOUSFIELD1990121}) that connective self-conjugate $K$-theory, $\ksc$ is the $\bb{E}_{\infty}$-ring defined as the connective cover of the homotopy fixed points of $\mo{KU}$ for the action of $\zz$ through complex conjugation $\psi^{-1}$.  In particular, since this $\zz$ action factors over a $\zz/2\zz$-action, there is a natural $\bb{E}_{\infty}$-algebra map $\ko\simeq \tau_{\geq 0}\mo{KU}^{h\zz/2\zz}\to \tau_{\geq 0}\mo{KU}^{h\zz}\simeq \ksc$.  This sits naturally in a cofiber sequence $\Sigma^2\ko\xrightarrow{\eta^2}\ko\to \ksc$, which is what allows for the methods developed in \textsection 5 to apply.

To begin, we define a filtration similar to what has been studied in \cite{angeltveit2015algebraic}, \cite{lee2023topological}.
\begin{definition}\label{def6.1}
There is an \textit{equivariant May-style filtration} on $\ksc_2$, given by the filtered $\bb{E}_{\infty}$-ring $\ksc_2^{fil}:=\tau^{1/2}_{\geq 0}((\tau_{\geq *}\mo{ku}_2)^{h\zz})$.  Similarly, one can define $\ko_2^{fil}:=\tau^{1/2}_{\geq 0}((\tau_{\geq *}\mo{ku}_2)^{h C_2})$.
\end{definition}
To begin, let's compute the associated graded objects for these equivariant May-style filtrations:
\begin{lemma}\label{lem6.2}
We have $$\pi_{*,*}(\ko_2^{gr})=\zz_{(2)}[v_1^2,\eta]/(2\eta),$$ with $|\eta|=(1,2)$, $|v_1|=(0,4)$.  The cofiber of $\eta^2$ on this filtered spectrum gives the underlying filtered spectrum $\ksc_2^{fil}$, and thus $$\pi_{*,*}(\ksc_2^{gr})=\zz_{(2)}[v_1^2,\eta,\rho]/(2\eta, \eta^2,\rho\eta, \rho^2),$$ with $|\rho|=(1,4)$.
\end{lemma}
\begin{proof}
Consider the filtered $\bb{E}_{\infty}$-ring $\tau_{\geq *}\mo{ku}_2$.  The associated graded is given simply by a polynomial ring $\zz_{(2)}[v_1]$, with $|v_1|=(0,2)$.  Since taking homotopy fixed points commutes with passing to the associated graded, we have $$((\tau_{\geq *}\mo{ku}_2)^{hC_2})^{gr}=(\pi_*\mo{ku}_2)^{hC_2}.$$  Computing the homotopy fixed points, it follows that  $$\pi_{*,*}((\pi_*\mo{ku}_2)^{hC_2})=\zz_{(2)}[v_1^2,x,y]/(2x,2y,y^2-xv_1^2),$$ where $|x|=(2,0)$ generates $\zz_{(2)}^{hC_2}$ in filtration degree 0, and $|y|=(1,2)$ generates the homotopy fixed points of $\zz_{(2)}\cdot v_1$ as a $\zz_{(2)}^{hC_2}$-module.  By the argument in \cite[~Lemma 2.16]{lee2023topological}, it suffices to apply $\tau^{1/2}_{\geq 0}$ to the associated graded.  Since, for $i\geq 0$ $\tau^{1/2}_{\geq 0}(x)_i=\tau_{\geq \lceil i/2\rceil}(x_i)$, one sees that $$\ko_2^{gr}=\tau_{\geq 0}^{1/2}((\pi_*\mo{ku})^{hC_2})=\zz_{(2)}[v_1^2,y]/(2y),$$
proving the first claim.

For $\ksc_2$, we proceed similarly, using that the homotopy fixed points for the trivial action are given by $\zz_{(2)}^{h\zz}\simeq \Lambda(w)$ on a class $w$ in $\pi_{-1}(\zz_{(2)}^{h\zz})$.  For the antipodal action, one can compute that $(\zz_{(2)}[t^{\pm 1}]/(t+1))^{h\zz}\simeq \Sigma^{-1}\zz/2\zz$.  Putting these together, $$\pi_{*,*}((\pi_*\mo{ku}_2)^{h\zz})=\zz_{(2)}[v_1^2,w,y]/(w^2,2y,y^2,wy),$$ with $|w|=(1,0)$, $|y|=(1,2)$.

In order to see that $\tau^{1/2}_{\geq 0}$ commutes with taking the associated graded for $\ksc_2$, one must compute $\pi_{*,*}(\tau_{\geq *}\mo{ku}_2)^{h\zz}$.  For this, note that our filtered object has the form $$\ldots \simto\tau_{\geq3} \mo{ku}_2\to\tau_{\geq 2}\mo{ku}_2\simto\tau_{\geq 1}\mo{ku}_2\to\mo{ku}_2\simto\ldots.$$  The homotopy fixed point spectral sequences computing $\pi_*(\tau_{\geq n}\mo{ku}_2^{h\zz})$ degenerate at the $E_2$-page for all $n$, which yields: $$\pi_{*,*}((\tau_{\geq *}\mo{ku}_2)^{h\zz})=\zz_{(2)}[\tau,z,\tau^{-2}\eta,\tau^{-4}v_1^2]/(2\tau^{-2}\eta,(\tau^{-2}\eta)^2,z^2,z\tau^{-2}\eta),$$ with $|\tau|=(-1,-1)$, $|z|=(1,0)$, $|\tau^{-2}\eta|=(1,2)$, $|\tau^{-4}v_1^2|=(0,4)$.  As there are no classes in bidegrees $(n+1,2n+1)$ for $n\geq 0$, and there is no $\tau$-torsion, it follows that $$\ksc_2^{gr}=\tau_{\geq 0}^{1/2}((\pi_*\mo{ku}_2)^{h\zz}).$$  By the computation above, this has homotopy groups $$\pi_{*,*}(\ksc_2^{gr})=\zz_{(2)}[v_1^2,\rho,\eta]/(\eta^2,2\eta,\rho^2,\eta\rho),$$ where $\rho$ comes from the class previously denoted $v_1^2w$, so $|\rho|=(1,4)$, and $\eta$ comes from the class previously denoted by $y$, sitting in bidegree $(1,2)$.
\end{proof}
\begin{corollary}\label{cor6.3}
The canonical map of graded $\bb{E}_{\infty}$-rings $\ko_2^{gr}\to\ksc_2^{gr}$ fits into a fiber sequence of $\ko_2^{gr}$-modules:
$$\Sigma^{2}\ko_2^{gr}\xrightarrow{\eta^2}\ko_2^{gr}\to \ksc_2^{gr}.$$
\end{corollary}
\begin{proof}
Note first that $\eta^2$ vanishes under this map, so we get some map from the cofiber of $\eta^2$ to $\ksc_2^{gr}$.  The easiest way to see that this is an equivalence is to work modulo 2, where, by modifying the computations of Lemma \ref{lem6.2}, one would find that the algebras $\ko_2^{gr}/2$ and $\ksc_2^{gr}/2$ are given by $\fff_2[v_1,\eta]$ and $\fff_2[v_1,\eta]/(\eta^2)$, respectively, which makes the computation clear.\footnote{Both associated graded algebras were $\bb{E}_{\infty}-\zz$-algebras, so the cofiber of $2$ has a canonical $\bb{E}_{\infty}$-algebra structure on the associated gradeds, even though $\ko_2/2$ itself cannot support any algebra structure.}
\end{proof}
Returning now to the main goal of computing $\THH(\ksc_2)$, we arrive at the main theorem of this section
\begin{theorem}\label{th6.4}
	There is an isomorphism $$\THH_*(\ksc_2)\simeq \THH_*(\ko_2,\ksc_2)\otimes_{\zz_{(2)}}\Gamma[\sigma^2\eta^2],$$ with $\sigma^2\eta^2$ a class in degree 4.
\end{theorem}
By corollary \ref{cor5.3}, the Whitehead spectral sequence computing $\THH_*(\ksc_2)$ has signature which can be identified with
$$E_1^{s,t}=\THH_{-s}(\ko_2,\pi_{t}(\ksc_2))\otimes_{\zz_{(2)}}\Gamma[x]\implies \THH_{t-s}(\ksc_2),$$ with the divided power class $x$ in bidegree $(-4,0)$.  Using similar methods as in \textsection 2, one can recover the fact from \cite{angeltveit2009topological} that $\THH_*(\ko_2,\zz_{(2)})$ is $\zz_{(2)}$ in degrees 0 and 5; $\zz/2^k\zz$ in degrees $r2^{k+2}-1$ and $r2^{k+2}-1+5$ for $r>0$ odd; and is zero otherwise.  Furthermore, $$\THH_*(\ko_2,\fff_2)\simeq \fff_2[u^4]\otimes\Lambda[u\xi_1^3,u^2\xi_2].$$\begin{center}
	\begin{sseqpage}[title={$E_1$-page of Whitehead spectral sequence in low degrees}, axes type = center,  classes = {draw = none },
		x range = {-8}{1}, y range = {0}{4}, x axis origin = {1}, y tick gap = {-0.5 cm}, x label = { s }, y label = { t }, x label style={ yshift = -15pt }, y label style={rotate=270, yshift=480pt, xshift=23pt}, class placement transform={scale=1}, xscale=1.7, yscale=1
		]
		\class["\zz_{(2)}"](0,0)
		\class["\zz_{(2)}x"](-4,0)
		\class["\zz_{(2)}"](-5,0)
		\class["\zz/2\zz"](-7,0)
		\class[" "](-8,0)
		\class["\zz_{(2)}x^{(2)}"](-8,0)
		\class["\zz_{(2)}"](0,1)
		\class["\zz_{(2)}"](-4,1)
		\class["\zz_{(2)}"](-5,1)
		\class["\zz/2\zz"](-7,1)
		\class[" "](-8,1)
		\class["\zz_{(2)}"](-8,1)
		\class["\zz_{(2)}"](0,4)
		\class["\zz_{(2)}"](-4,4)
		\class["\zz_{(2)}"](-5,4)
		\class["\zz/2\zz"](-7,4)
		\class[" "](-8,4)
		\class["\zz_{(2)}"](-8,4)
		\class["\zz/2\zz"](0,3)
		\class["\zz/2\zz"](-4,3)
		\class["\zz/2\zz"](-5,3)
		\class["\zz/2\zz"](-7,3)
		\class[" "](-8,3)
		\class["\zz/2\zz^{\oplus 2}"](-8,3)
	\end{sseqpage}
\end{center}  In particular, when we run the Whitehead spectral sequence for $\THH(\ksc_2)$, the class $\sigma^2\eta^2$ in bidegree $(-4,0)$ cannot hit anything for degree reasons (and since $\ksc_2\to \THH(\ksc_2)$ must split as $\ksc_2$ is an $\bb{E}_{\infty}$-ring), and is thus a permanent cycle, so too then are all powers of $\sigma^2\eta^2$.  The claim will reduce to showing that $x$ and all of its divided powers are permanent cycles.  To aid in this endeavor, we investigate the equivariant May-type spectral sequence arising from $\THH(\ksc_2^{fil})$, with signature
$$E_1^{s,t}=\pi_{*,*}\THH(\ksc_2^{gr})\implies \THH_*(\ksc_2).$$
\begin{proposition}\label{prop6.5}
There is an isomorphism of graded algebras $$\THH_*(\ksc_2^{gr})\simeq \THH_*(\ko_2^{gr},\ksc_2^{gr})\otimes_{\zz_{(2)}} \Gamma[\sigma^2\eta^2],$$ with $\sigma^2\eta^2$ a class in degree 4.
\end{proposition}
\begin{proof}
As a consequence of corollary \ref{cor5.3} and corollary \ref{cor6.3}, the Whitehead spectral sequence computing $\THH_*(\ksc_2^{gr})$ has signature $$E_1=\THH_*(\ko_2^{gr},\pi_*(\ksc_2^{gr}))\otimes_{\zz_{(2)}}\Gamma[\sigma^2\eta^2]\implies \THH_*(\ksc_2^{gr}).$$  It suffices to show that the classes in $\Gamma[\sigma^2\eta^2]$ are permanent cycles, since then there will be an algebra map $\THH_*(\ko_2^{gr},\ksc_2^{gr})\otimes_{\zz_{(2)}}\Gamma[\sigma^2\eta^2]\to \THH_*(\ksc_2^{gr})$ inducing the desired isomorphism.

To see this, we will work with the Bockstein spectral sequence associated to the filtered $\bb{E}_{\infty}$-$\zz$-algebra $\THH(\ksc_2^{gr}\otimes_{\zz}\zz_{(2)}^{bok})$.  This spectral sequence has signature $$E_1^{s,t}=\THH_*(\fff_2[\eta,v_1,\tilde{v_0}]/(\eta^2))\implies \THH_*(\ksc_2^{gr})_2^{\wedge}.$$  By monoidality of $\THH$, we have $$\THH(\fff_2[\eta,v_1,\tilde{v_0}]/\eta^2)\simeq \THH(\fff_2[\eta]/\eta^2)\otimes_{\THH(\fff_2)}\THH(\fff_2[v_1])\otimes_{\THH(\fff_2)}\THH(\fff_2[\tilde{v_0}]).$$  Using that $\fff_2[\eta]/\eta^2$ is a square zero extension of $\fff_2$ in $\bb{E}_{\infty}$-$\fff_2$-algebras, we find that $$\THH_*(\fff_2[\eta]/\eta^2)=\THH_*(\fff_2)\otimes_{\fff_2}\fff_2[\eta]/\eta^2\otimes\Lambda[\sigma\eta]\otimes_{\fff_2} \Gamma[\sigma^2\eta^2].$$  From this, it follows that $$\THH_*(\fff_2[\eta,v_1,\tilde{v_0}]/\eta^2)=\THH_*(\fff_2)\otimes_{\fff_2}\fff_2[\eta,v_1,\tilde{v_0}]/\eta^2\otimes_{\fff_2}\Lambda[\sigma\eta,\sigma v_1,\sigma \tilde{v_0}]\otimes_{\fff_2} \Gamma[\sigma^2 \eta^2].$$
In the 2-Bockstein spectral sequence, the multiplicative generators have bidegrees $|\tilde{v_0}|=(1,1)$, $|\sigma\tilde{v_0}|=(0,1)$, $|\sigma\eta|=(-2,0)$, $|\sigma v_1|=(-3,0)$, $|\eta|=(-1,0)$, $|v_1|=(-2,0)$,  $\sigma^2\eta^2=(-4,0)$, and the class $u$ in bidegree $(-2,0)$ which generates $\THH_*(\fff_p)$ as a polynomial algebra.  We begin by examining the class $\sigma^2\eta^2$ in degree $(-4,0)$.  Almost all of the classes in the $\tilde{v_0}$-tower on the class $\sigma^2\eta^2$ must survive the spectral sequence in order to give the $\zz_{(2)}\cdot \sigma^2\eta^2$ class in degree 4 of $\THH_*(\ksc_2^{gr})$.  One finds that $\sigma^2\eta^2$ cannot support any differentials, since any nontrivial differential on this class would kill the entire tower.  We make use of the following lemma.

\begin{lemma}\label{lem6.6}
If $a$ is a class on the $E_k$-page of the Bockstein spectral sequence computing $\THH(\ksc_2^{gr})_2^{\wedge}$, with $a\neq 0$, but $\tilde{v_0}a=0$, then $\deg^{t}(a)\leq k-1$, where $\deg^t(-)$ denotes the integer $n$ such that $a\in E_k^{s,n}$ for some $s$.
\end{lemma}
\begin{proof}  For $k=1$, this is vacuous, since there is no $\tilde{v_0}$-torsion on the $E_1$-page of this spectral sequence.  We proceed by induction.  Suppose that $a\neq 0$ is a class on the $E_k$-page with $\tilde{v_0}a=0$, and $\deg^{t}(a)\geq k$.  The fact that $\tilde{v_0}a=0$ means that at some point earlier in the spectral sequence, say on the $E_{k-i}$-page ($i>0$), we had a class $b$ with $d_{k-i}(b)=\tilde{v_0}a$.  The class $b$ then necessarily has $t$-degree $i+1>1$.  In particular, $\tilde{v_0}$ must divide $b$ for degree reasons, so $b=c\tilde{v_0}$, for some class $c$ (or more accurately, comes from a class on $E_1$ divisible by $\tilde{v_0}$, and by our inductive hypothesis, there cannot be a differential taking $c$ to a nonzero class which multiplies with $\tilde{v_0}$ to zero).  Now, $d(c)\neq a$, but $d(\tilde{v_0}c)=d(b)=\tilde{v_0}a$, so that $a-d(c)\neq 0$, but $\tilde{v_0}(a-d(c))=0$.  Since $a-d(c)$ is a $\tilde{v_0}$-torsion class on $E_{k-i}$ with $t$-degree $k\geq k-i$, it must be 0 by induction, so that $a=d(c)$, contradicting the choice of $a$.
\end{proof}
This lemma implies that if the differential of any class in $t$-degree $0$ or $1$ is nontrivial, then the entire $\tilde{v_0}$-tower on that class dies.  Let $n$ be the smallest natural number such that $(\sigma^2\eta^2)^{(2^n)}$ does not live in $\THH_*(\ksc_2^{gr})$.  In particular, we must have that the entire $\tilde{v_0}$-tower on the analogous class in the mod 2 Bockstein must vanish.  Since $(\sigma^2\eta^2)^{(2^{n-1})}$ squares to a torsion-free class, there must be a nonvanishing $\tilde{v_0}$ tower in total degree $2\cdot 2^{n-1}\cdot 4=2^{n+2}$ in the Bockstein spectral sequence.  Considering the map from the spectral sequence associated to $\THH(\zz_{(2)}^{fil})$ shows that the classes divisible by $u$ and $\sigma\tilde{v_0}$ are all $\tilde{v_0}$-torsion.  In order for $\eta$ to be $2$-torsion, we need a differential to hit $\eta\tilde{v_0}$, and this can be checked to come from $\sigma\eta$.  Thus, the only classes that can contribute to a nonvanishing $\tilde{v_0}$-tower are multiples of the classes $(\sigma^2\eta^2)^{(2^k)}$, for $k<n$, powers of $v_1$, and $\sigma v_1$.  Since no power of $v_1$ divides any element of $\zz_{(2)}[\sigma^2\eta^2]$, the only contribution can come from $\sigma v_1$ and the $(\sigma^2\eta^2)^{(2^k)}$.  The total degree of $\sigma v_1\cdot \prod_{k<n}(\sigma^2\eta^2)^{(2^k)}$ is $2^{n+1}$, but the tower we need is in total degree $2^{n+2}$, and thus must come from $(\sigma^2\eta^2)^{(2^n)}$! This shows that all of our $(\sigma^2\eta^2)^{(2^n)}$ classes have to survive this Bockstein spectral sequence, proving that they survive to give the divided power classes in $\THH_*(\ksc_2^{gr})$ that we were looking for.
\end{proof}

\begin{proof}[Proof of Theorem \ref{th6.4}]
By proposition \ref{prop6.5}, the equivariant May-type spectral sequence has $E_1$-page $$\THH_*(\ko_2^{gr},\ksc_2^{gr})\otimes_{\zz_{(2)}}\Gamma[\sigma^2\eta^2].$$  The class $\sigma^2\eta^2$ sits in $(s,t)$-degree $(0,4)$, and the only classes in $\THH_*(\ko_2^{gr},\ksc_2^{gr})$ with positive $s$-degree are the classes $\rho$ and $\eta$, in degrees $(1,4)$ and $(1,2)$, respectively.  Thus, the classes in $\Gamma[\sigma^2\eta^2]$ are permanent cycles.  By comparing with the map $\THH_*(\ksc_2^{gr})\to\THH_*(\ksc_2^{gr},\zz_{(2)})$, we find that $\sigma^2\eta^2$ reduces to the same class originally termed $x$ in the Whitehead spectral sequence.  In particular, this analysis shows that the classes $\Gamma[x]$ in the Whitehead spectral sequence for $\THH(\ksc_2)$ must be permanent cycles as well, so that the $E_{\infty}$-page of the Whithead spectral sequence is the tensor product of the $E_{\infty}$-page of the Atiyah-Hirzebruch spectral sequence computing $\THH(\ko_2,\ksc_2)$ with the divided power algebra $\Gamma[x]$.  We find that there is an induced map of graded-commutative $\zz_{(2)}$-algebras $$\THH_*(\ko_2,\ksc_2)\otimes_{\zz_{(2)}}\Gamma[\sigma^2\eta^2]\to\THH_*(\ksc_2),$$ which, from the computation of the Whitehead spectral sequence for $\THH(\ksc_2)$, must be an isomorphism.
\end{proof}

\begin{remark}\label{rem6.7}
	$\THH_*(\ko_2,\ksc_2)$ can be computed as a graded abelian group from the work of \cite{angeltveit2009topological}, noting that they prove that $\eta^2$ acts as zero on $\overline{\THH}_*(\ko_2)$, which determines $$\overline{\THH}(\ko_2,\ksc_2)\simeq\cofib(\eta^2:\Sigma^2\overline{\THH}(\ko_2)\to\overline{\THH}(\ko_2))$$ up to extension problems.  Since the only classes in $\overline{\THH}_*(\ko_2)$ in odd degrees are copies of $\zz_{(2)}$ living in what \cite{angeltveit2009topological} call $F^{\ko}$, there can only be nontrivial extension problems if the map from $$\THH_{5+4n}(\ko_2)\to\THH_{5+4n}(\ko_2,\ksc_2)$$ is not surjective on the torsion-free parts.  However, we know from the computations in \cite[~\textsection 7.2-7.3]{angeltveit2009topological}  that $$\THH_{5+4n}(\ko_2)\to\THH_{5+4n}(\ko_2,\mo{ku}_2)$$ induces an isomorphism on the torsion-free part, and this map factors as $$\THH_{5+4n}(\ko_2)\to\THH_{5+4n}(\ko_2,\ksc_2)\to\THH_{5+4n}(\ko_2,\mo{ku}_2).$$  Thus, we find that, as a graded abelian group, $$\THH_*(\ko_2,\ksc_2)\simeq \THH_*(\ko_2)\oplus \THH_{*-3}(\ko_2).$$  Combined with the above, we have completely determined $\THH_*(\ksc_2)$.
\end{remark}

\newpage
\printbibliography

\end{document}